\documentclass[12pt]{article}
\usepackage[latin1]{inputenc}
\usepackage{amsfonts,epsfig,graphicx}
\usepackage{amsmath,amssymb,amsthm}
\usepackage{epsf}
\usepackage{graphics}
\usepackage{amsfonts}
\usepackage{amsmath}
\usepackage{amssymb}
\usepackage[usenames,dvipsnames]{color}
\usepackage{graphicx}
\usepackage{times}

\usepackage{tikz}

\usepackage{latexsym}
\usepackage{framed}
\usepackage{caption}

\newcommand{\mbbE}{\mathbb{E}}

\newcommand{\mbbP}{\mathbb{P}}
\newcommand{\mbbR}{\mathbb{R}}

\newcommand{\mcB}{\mathcal{B}}
\newcommand{\mcC}{\mathcal{C}}
\newcommand{\mcM}{\mathcal{M}}
\newcommand{\mcY}{\mathcal{Y}}

\newcommand{\barrho}{\bar{\rho}}

\newcommand{\bpf}{\begin{proof}}
\newcommand{\epf}{\end{proof}}

\theoremstyle{definition}\newtheorem{thm}{Theorem}
\theoremstyle{definition}\newtheorem{lem}[thm]{Lemma}
\theoremstyle{definition}
\theoremstyle{definition}\newtheorem{exm}[thm]{Example}
\theoremstyle{definition}\newtheorem{defi}[thm]{Definition}
\theoremstyle{definition}\newtheorem{rem}[thm]{Remark}
\theoremstyle{definition}

\usepackage{lineno}

\begin{document}




\title{
Nash equilibrium structure of\\
Cox process Hotelling games }

\author{
Venkat Anantharam\\ 
{\em EECS Department}\\
{\em University of California, Berkeley}\\
ananth@eecs.berkeley.edu\\
 \and
Fran\c{c}ois Baccelli\\
{\em Department of Mathematics and ECE Department}\\
{\em University of Texas, Austin}\\
and\\
{\em INRIA/ENS},
{\em Paris, France}\\
francois.baccelli@austin.utexas.edu {\em or} francois.baccelli@inria.fr}

\maketitle

\begin{abstract}
We study an $N$-player game where a pure action of each player is to
select a non-negative function on a Polish 
space supporting a finite diffuse measure, subject to a finite constraint on the
integral of the function. This function is used to define the intensity of a Poisson point process on the Polish space.
The processes are independent over the players, and the value to a player is the measure of the union of its 
open
Voronoi cells in the superposition point process. 
Under randomized strategies, the process of points of a player is thus a Cox process, and the nature of competition between the players is akin to that in Hotelling competition games.
We characterize when such a game admits Nash equilibria and prove that when a Nash equilibrium exists, it is unique and comprised of pure strategies that are proportional in the same proportions as the total intensities. We give examples of such games where Nash equilibria do not exist. A better understanding of the criterion for the existence of Nash equilibria remains an intriguing open problem.

\end{abstract}


\noindent
{\bf Classification}: {60G55;91A06;91A60}


\section{Introduction}		\label{sec:Intro}

\subsection{Informal problem formulation}		
\label{subsec:Descr}

We study {\emph {games of spatial competition of the Hotelling type}},
where $N$ players
compete for space. Here space is modeled
as a complete separable metric space (i.e. a {\emph {Polish space}}) $D$ supporting a {\emph {diffuse}} 
(i.e. non-atomic) finite positive measure $\eta$ on its
Borel $\sigma$-field, which we denote by $\mcB$.
Throughout the paper we think of $D$ as being endowed with a fixed metric $d$ generating its topology.
The game in question is a one-shot game. Each player's
action consists in selecting a nonnegative measure on 
$D$ 
\footnote{Since we will always consider $D$ as being metrized by the metric $d$ and endowed with its Borel $\sigma$-algebra
$\mcB$, we will not mention $d$ and $\mcB$ 
where these can be inferred from the context.
Thus, for instance, by a nonnegative measure on $D$ what we actually mean is a nonnegative measure on 
$(D, \mcB)$. Similarly, when we talk about the 
open Voronoi cell of a point in a configuration of points from $D$, we implicitly mean that the Voronoi cells are based on the metric $d$.}
which is 
absolutely continuous with respect to $\eta$,
among the set of all measures with a fixed finite and positive total
mass. This mass constraint depends on the player
and represents the total available intensity that the player can 
deploy over the space. 
The measure chosen by each player
results in a Poisson process of points on $D$
with this measure as its intensity measure, 
these processes being independent.
Thus, if a player uses a randomized strategy, the point process on $D$ generated by each player is a {\emph {Cox point process}},
namely a point process which is conditionally Poisson given
its random intensity.
The payoff of 
each player 
is the total $\eta$ measure
of the 
union
of the {\emph {open Voronoi cells}} of the points of the point process of this player, where the
open Voronoi cells are evaluated with respect to the superposition of the point processes of all players (if there are no points, which can happen with 
positive probability, each player gets zero value; further, we need to make a technical assumption on the metric structure of $D$ which ensures that the union of the open Voronoi cells is of full measure, conditioned on there being at least one point in the superposition point process).
Our goal in this paper is to study the {\emph {Nash equilibria}} of this game. 

In the $2$-player case, one motivation for such a formulation comes from
a model for defense against threats. The underlying space may be thought
of as a model for the set of possible attack modalities, with the metric
indicating how similar attacks are to each other. The defender and 
the attacker are respectively interested in defending against or deploying 
the different kinds of attacks, and evaluating value by the 
total $\eta$ measure of the union of the open Voronoi
cells of the points of each of them evaluates how well each of them is doing with regard to its individual objective of getting the upper hand over the other.
The stochastic nature of the placement of the points 
of a player is meant to capture the idea that the deployment of effort
only results in success in a stochastic way. The finite total intensity that
each each player can deploy represents individual budget constraints.
The study of Nash equilibria is then motivated by the goal of getting
some insights into how the individual players (i.e. the attacker and defender) might play when faced with
such an environment. Apart from its possible intrinsic interest, it turns
out that it is no more difficult to study the $N$-player version of the
game than the $2$-player version, for reasons that we will soon see,
so we have formulated the problem we study at this apparently broader level
of generality.

\subsection{Literature survey}		\label{subsec:Backgd}


The origin of the study of spatial competition models is generally
attributed to a paper by Hotelling \cite{Hotel}.
Hotelling's model 
has some additional features, such as 
prices set by the sellers,
which 
will not play a role in our formulation. Building on the {\emph {Cournot 
duopoly model}} \cite{Cournot}, \cite[Sec. 27.5]{Varian}, as refined by 
Bertrand \cite{Bertrand}, \cite[Sec. 27.9]{Varian} and Edgeworth
\cite{Edgeworth}, the key innovation of \cite{Hotel} is to introduce
spatial aspects to the 
modeling of the competition between sellers for buyers.
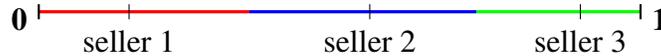
\begin{figure}
\begin{center}
\begin{tikzpicture}[xscale=8]
\draw[-][draw=red, very thick] (0,0) -- (0.35,0);
\draw[-][draw=blue, very thick] (0.35,0) -- (0.7275,0);
\draw[-][draw=green, very thick] (0.7275,0) -- (1,0);
\draw [thick] (0,-.1) node[left]{\bf 0} -- (0,0.1);
\draw [thick] (1,-.1) node[right]{\bf 1} -- (1,0.1);
\draw [-] (0.15,-.1) node[below]{seller 1} -- (0.15,0.1);
\draw [-] (0.55,-.1) node[below]{seller 2} -- (0.55,0.1);
\draw [-] (0.9,-.1) node[below]{seller 3} -- (0.9,0.1);
\end{tikzpicture}
\caption{Hotelling competition on the unit interval with three sellers, with locations indicated. The individual open Voronoi cells are indicated in color. In the absence of price discrimination, consumers residing at a location in the open Voronoi cell of a seller will go to that seller.}
    \label{fig:Hotelling}
\end{center}
\end{figure}
Specificially, Hotelling 
considers the problem faced by two sellers as to where to position 
themselves along an interval of fixed length, which models a market 
along which consumers are uniformly distributed, and how to individually 
set prices for the one identical good that they sell so that each 
seller maximizes its profit, given the strategy of the other seller.
In the model of Hotelling 
each consumer incurs a transportation cost proportional
to its distance from the seller from which it buys, which then determines
which seller it prefers, given the location of the sellers. 

Notice that for fixed locations of the two sellers and 
if price discrimination is not possible (i.e. prices are identical
at the two sellers) consumers would go to the closer seller, so
the problem of each seller becomes that of how to position itself, in 
reaction to position of the other seller, so as to maximize the 
length of its Voronoi cell in the {\emph {Voronoi decomposition}} associated to the
locations of the two sellers. See \cite[Fig. 1, pg. 45]{Hotel}.
The problem considered in this paper is of this purely 
Voronoi-cell-decomposition type. 
It is worth noting that the solution concept in 
\cite{Hotel} or, for that matter, in \cite{Cournot}
is already of the type that one would today call a pure-strategy 
Nash equilibrium, though of course these works predate 
by several decades the work of 
von Neumann and Morgenstern \cite{vNM} and 
Nash \cite{NashPNAS, NashAnnals}, which formalized the notion of 
Nash equilibrium.

There is by now a vast literature on the Hotelling competition model,
and many different variants have been developed and studied.
Rather than attempt to survey
this work, particularly since the precise problem 
formulation we consider does not appear in the prior literature 
in this area, we refer the reader to the surveys 
in \cite{Graitson, GabThisse, ELT, ReVelle},
and to the papers in the recent 
edited volume \cite{Mallozi}.

Another growing body of work that is related 
to the themes of this paper is that of
Voronoi games, see the seminal papers 
\cite{Ahn04} and \cite{Cheong04}.
Here also the players are competing to capture regions of space
according to the Voronoi decomposition of the underlying space
based on the choice of the locations of their points but, in contrast
to the model we consider, the players are assumed to have
control over exactly where they can place their points with the constraint, in some
of the literature on Voronoi games, that
these locations should lie in a given finite set of potential
locations in the ambient space. Further, in contrast to the model we consider, much of
this literature assumes that the players place their points one at a time,
alternating between the players.
There is also a particular interest in this literature in the study
of Voronoi games on graphs, see e.g. \cite{Teramoto, DurrThang}.
Note that this is
incompatible with the non-atomic nature of the underlying measure
$\eta$ required for our formulation.
Nevertheless, our formulation was partly inspired by a recent work in this
area, on so-called {\emph {Voronoi choice games}}, 
by Boppana et al. \cite{BHMM}.

There are also some similarities between the game we analyze and the
study of the so-called {\emph {Colonel Blotto games}}, see 
\cite{Col-Blot},
on variants of which also there is a rapidly growing literature.
In the game we study, the ability of a player to control exactly where
to place its points, as in the literature on Hotelling games or 
Voronoi games,
can be thought of as being replaced
by a softer 
ability of being able to control the intensity or, in effect, 
the local mean number of points, in the same way as the 
formulation of the so-called {\emph {General Lotto game}},
see \cite{Gen-Lot},
softens
the ability of individual players to fix the number
of soldiers to be placed on each battlefield of the Colonel Blotto game. 
However, the reward structure of the players in a Colonel Blotto game
is completely different from that in the Hotelling games or Voronoi 
games.

\subsection{Notational conventions}		\label{subsec:Notation}

$\mathbb{N}$ denotes the set of natural numbers.
$\mathbb{R}_+$ denotes the set of non-negative real numbers.
We use $:=$ and $=:$ for equality by definition.
The indicator of a set or an event $E$ is written
as $1_E$ or $1(E)$.

\subsection{Structure of the paper}

Section \ref{sec:CPHG} 
sets up the structure of $2$-player Cox processes Hotelling games.
Section \ref{sec:Lemmas} 
introduces the 
$N$-player Cox process Hotelling games
and develops several general properties
of the value function in these games;
these are used later to prove the main results.
Section \ref{sec:Homogeneous} is focused on the case where the underlying Polish space on which the game is played is compact and admits a transitive group of metric preserving automorphisms, with the base measure of the game being invariant under this group.
Section \ref{sec:General} studies the general case of the $N$-player game
and determines the structure of
the Nash equilibria when they exist.
A necessary and sufficient condition for the existence of Nash equilibria is provided, and also examples are given 
where no Nash equilibrium exists.
Although
Nash equilibrium, when it exists, is unique and comprised of pure strategies,
it is established that Cox process Hotelling games are not ordinal potential games in general.
A family of so-called restricted Cox process Hotelling games is defined to provide a vehicle to better understand the meaning of the criterion for the existence of Nash equilibria.

\section{Cox process Hotelling games}	\label{sec:CPHG}



\subsection{Diffuse non-conflicting finite positive measures}
\label{subsec:diffuse}

Let $D$ be a 
Polish space. We write $\mcB$ for the Borel $\sigma$-field of $D$.
A finite positive measure $\eta$ on $D$
is called {\emph {diffuse}}
 if for every Borel set $B \in \mcB$ with $\eta(B) > 0$
there is a Borel subset $C \subset B$ with $0 < \eta(C) < \eta(B)$.
One can define 
a Poisson process on 
$D$ based on such a diffuse positive measure $\eta$ on $D$,
see e.g. \cite[Ch. 9]{DvJ2}.
A point process 
on $D$ 
is a random counting measure on
$D$.
The Poisson process on $D$ with intensity measure $\eta$, where $\eta$ is any finite positive diffuse
measure on $D$, is the point process obtained by first selecting the total number of points
according to the Poisson distribution on $\mathbb N$ with mean $\eta(D)$, and then
sampling independently the location of the points, if any, according to the measure $\eta(.)/\eta(D)$ on $\mcB$.

Recall that $d$ denotes the fixed metric on $D$ generating its topology.
The {\emph {open Voronoi cells}} of any point process $\Phi$ with respect to $d$ can be defined in the usual way. Namely, for
$x$ in the support of the point process $\Phi$,
the open Voronoi cell $W_{\Phi}(x)$ is comprised of those $z \in D$
such that $d(x,z) < d(x',z)$ for all points $x' \neq x$ in the support of $\Phi$. 
For our purposes, we need to impose on $\eta$ the condition
that the union of the open Voronoi cells of the Poisson process on $D$
with intensity measure $\eta$ has full measure $\eta(D)$ with 
probability $1$, conditioned on there being at least one point in this Poisson process. If $\eta$ satisfies this condition, we call it
{\emph {non-conflicting}}. A sufficient condition for this to hold is given in Appendix \ref{sec:append}.

The importance of imposing a non-conflicting condition can be
understood by considering the following example.

\begin{exm}		\label{exm:conflicting}

Suppose $D$ is the disjoint union of two unit intervals
$I$ and $J$, each of length $1$.
Assume that every point in $I$ is at distance $2$ from every
point in $J$. 
\begin{figure}
\begin{center}
\begin{tikzpicture}[xscale=1.5]
\draw[-][draw=red, ultra thick] (-2.0,-1) -- (-2.0,1);
\draw[-][draw=red, ultra thick] (2.0,-1) -- (2.0,1);
\node at (-2.6,0) {length 1};
\node at (2.6,0) {length 1};
\draw[->] (-2.6,0.25) -- (-2.6,1);
\draw[->] (-2.6,-0.25) -- (-2.6,-1);
\draw[->] (2.6,0.25) -- (2.6,1);
\draw[->] (2.6,-0.25) -- (2.6,-1);
\node at (0,0) {distance 2};
\draw[->] (0.5,-0.25) -- (1.8,-0.9);
\draw[->] (-0.5,0.25) -- (-1.8,0.9);
\node at (0,-1) {distance 2};
\draw[->] (0.7,-1) -- (1.8,-1);
\draw[->] (-0.7,-1) -- (-1.8,-1);
\end{tikzpicture}
\caption{A metric space comprised of two intervals, each of length $1$. Every point in the interval on the left is at distance $2$ from every point in the interval on the right. The metric on each interval is the usual one.}
    \label{fig:Twointervals}
\end{center}
\end{figure}
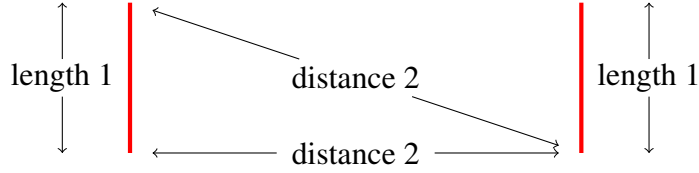
It is straightforward to check that 
the resulting metric makes $D$ a 
Polish space and that the measure $\eta$ on the Borel $\sigma$-field
of $D$ which corresponds to Lebesgue measure on the two unit intervals
comprising $D$ is a finite positive diffuse measure. However,
$\eta$ is not non-conflicting.
This is because, for instance, on the event of positive probability
that 
the Poisson process on $D$ with intensity measure $\eta$ has two points in $I$
and no points in $J$, every point of $J$
will be distance $2$ from each of the two points of the Poisson process
in $I$ and so 
does not belong to the open Voronoi cell of either point.
\hfill $\Box$
\end{exm}

Throughout the paper, $\eta$ is
a finite positive diffuse 
non-conflicting measure on $D$ which will be referred to as the {\emph {base measure}}.
It is straightforward to check that if $\nu$ is any finite positive
measure on $D$
that is absolutely
continuous with respect to $\eta$, then $\nu$ is also
diffuse and non-conflicting.
We write $\mcM(D)$ for the set of nonnegative finite measures on 
$D$.

\subsection{Radon-Nikodym derivatives}		\label{subsec:RN}

In the {\emph {Cox process Hotelling games}} 
considered in this paper, we think
of each individual player as choosing a positive measure 
on $D$ of fixed total 
mass 
that is 
absolutely continuous with respect to the base measure $\eta$.
The measure chosen by the player then 
serves as the intensity measure for a Poisson process of points
on $D$,
which we think of as the points belonging to that player.
By the Radon-Nikodym theorem, we may identify the measure chosen
by the player with its likelihood function with respect to the base measure.

With this viewpoint in mind,
for any $\rho > 0$, let 
\begin{equation}		\label{eq:Crho}
\mcC(\rho) := \{f:D \to \mbbR_+ \mbox{ s.t. } \int_D f(x) \eta(dx) = \rho \}.
\end{equation}
$\mcC(\rho)$ can be thought of as a subset of
$L^1(\eta)$, i.e. the Lebesgue space of $\mcB$-measurable functions on 
$D$ that are absolutely integrable with respect to the base
measure $\eta$, and also as a set of measures on $D$
by identifying $f \in \mcC(\rho)$ with the measure $f \eta$. With the latter viewpoint in mind, we think of $\mcC(\rho)$ as endowed with the 
topology of weak convergence of measures
\cite{Partha},
which it inherits as a subset of $\mcM(D)$. Note that,
for all $\rho > 0$ and
all $f \in \mcC(\rho)$, the measure $f \eta$ is diffuse and
non-conflicting. Also note that, for all $\rho > 0$, the 
set $\mcC(\rho)$ is a convex subset of $\mcM(D)$. However, since we have endowed $\mcC(\rho)$ with the topology of weak convergence of measures on $D$, this set is not closed.

For any $\rho > 0$, the subset of $\mcM(D)$ 
comprised of measures of total mass $\rho$,
with the topology
of weak convergence, can be metrized
so as to make it
a complete 
separable metric space \cite[Thm. 6.2 and Thm. 6.5]{Partha}.
As a metric space it is first countable, so
it
suffices to discuss convergence of sequences rather than of nets 
\cite[Thm. 8]{Kelley}.
 To discuss randomized strategies
of the individual players we need to be able to discuss probability distributions on subsets of
the type $\mcC(\rho)$ of $\mcM(D)$.
This is made possible by the following result.

\begin{lem}	\label{lem:AbsContMeas}
For each $\rho > 0$, the set  $\mcC(\rho)$ is a Borel subset of 
$\mcM(D)$ when $\mcM(D)$ is endowed with the topology of weak 
convergence.
\end{lem}

\bpf
See \cite[Thm. 3.5]{Lange}.
\epf


\subsection{$2$-player games}		\label{subsec:2-player}

In the $2$-player version of the game, 
a pure action of Alice 
consists in choosing an intensity measure which is absolutely continuous 
with respect to the
measure $\eta$. Namely 
Alice chooses 
a nonnegative measurable function
$f_A:D\to \mathbb{R}_+$
as the Radon Nikodym derivative with respect to $\eta$ of the intensity measure of its Poisson point process.
We denote the set of functions from which 
this choice must be made 
as $\mcC_A$, i.e.
\begin{equation}			\label{eq:classA}
\mcC_A := \mcC(\rho_A) =  \{f_A:D \to \mbbR_+ \mbox{ s.t. } \int_D f_A(x) \eta(dx) = \rho_A \}.
\end{equation}
Thus $\rho_A$ can be thought of as a constraint
on the total intensity that Alice can deploy for her point process.

Similarly, Bob chooses a nonnegative measurable function $f_B:D\to \mathbb{R}_+$
within the class of functions $\mcC_B$, where
\begin{equation}		\label{eq:classB}
\mcC_B := \mcC(\rho_B) = \{f_B:D \to \mbbR_+ \mbox{ s.t. } \int_D f_B(x) \eta(dx) = \rho_B \}.
\end{equation}
Here $\rho_A$ and $\rho_B$ are fixed positive constants.
Note that $\mcC_A$ and $\mcC_B$ are convex subsets of 
$\mcM(D)$, but neither of them is closed
in the topology of weak convergence on 
$\mcM(D)$.

We denote such a $2$-player Cox process Hotelling game by 
$(D,\eta, \rho_A, \rho_B)$.

Assume that Alice plays $f_A$ and Bob plays $f_B$. The resulting
value for Alice, denoted by $V_A(f_A,f_B)$, is defined as follows.
Let $\Phi_A$ and $\Phi_B$ be independent Poisson point processes on 
$D$
with intensity measures $f_A \eta$ and $f_B \eta$, respectively.
We think of the points of $\Phi_A$ as the points of Alice, since these
result from the choice of $f_A$ by 
Alice. Similarly we think of the points of $\Phi_B$ as the points of Bob.
Let $\Phi :=\Phi_A+\Phi_B$.
For all $x \in \Phi$, let $W_\Phi(x)$ denote the open Voronoi cell of $x$ 
with respect to
$\Phi$.
Then, 
$$V_A(f_A,f_B) := \mathbb{E} [\sum_{x\in \Phi_A} \eta(W_\Phi(x))],$$
where the 
expectation is with respect to the joint law of $\Phi_A$ and $\Phi_B$,
which are 
assumed to be independent.
Here, by definition, a sum over an empty set is 0.
The value for Bob resulting from this pair of actions, 
denoted $V_B(f_A,f_B)$, is determined
similarly, i.e. it is
$$V_B(f_A,f_B) := \mathbb{E} [\sum_{x\in \Phi_B} \eta(W_\Phi(x))].$$

In words, the value of Alice is the mean value of the
sum of the $\eta$-measures of the open $\Phi$-Voronoi cells centered at her points
if she has points, and is $0$ otherwise, and likewise for the value of Bob. Intuitively, one thinks of the 
points of Alice and those of Bob as competing to capture the ambient space,
with a point in the ambient space $D$ belonging to Alice if the closest
point to it is one of Alice, and to Bob otherwise. Under our non-conflicting condition on the base measure $\eta$, there is no need to worry about how to
break ties.

Note that 
\begin{equation}		\label{eq:conservation2}
V_A(f_A,f_B) + V_B(f_A,f_B) = \eta(D) (1 - e^{- \rho }),
\end{equation}
for all $(f_A,f_B) \in \mcC_A \times \mcC_B$,
where $\rho := \rho_A + \rho_B$. This is because, by virtue of our assumption that $\eta$ is non-conflicting, the 
sum of the values of Alice and Bob is $\eta(D)$, except on the event where
$\Phi_A$ and $\Phi_B$ both do not have points, which is an event
of probability $1 - e^{- \rho }$.
This formula in Equation \eqref{eq:conservation2} is 
formally established in Lemma \ref{lem:totalvalue} below.


\subsection{Nash equilibrium in $2$-player games}
\label{subsec:NE2-defn}

In view of Lemma \ref{lem:AbsContMeas},
any mixed strategy pair in the $2$-player game 
$(D, \eta, \rho_A, \rho_B)$ between Alice and Bob can 
be written as 
$(f_A(M_A), f_B(M_B))$, where 
$M_A \in \mcM_A$ and $M_B \in \mcM_B$ are independent
random variables representing the randomizations used by Alice and
Bob respectively in implementing randomized strategies, and
$f_A(m_A) \in \mcC_A$ (respectively $f_B(m_B) \in \mcC_B$) is the 
choice of Alice (respectively Bob) in case the realization of
her random variable $M_A$ is $m_A$ (respectively, the realization of
his random variable $M_B$ is $m_B$).
The value of Alice in such a mixed strategy is 
$\mbbE[ V_A(f_A(M_A), f_B(M_B))]$, while that of Bob is 
$\mbbE[ V_B(f_A(M_A), f_B(M_B))]$, the expectations being taken
with respect to the joint distribution of $M_A$ and $M_B$, which are
independent. 
As a direct consequence 
of Equation \eqref{eq:conservation2},
we have
\begin{equation}		\label{eq:randconserve2}
\mbbE[ V_A(f_A(M_A), f_B(M_B))] + \mbbE[ V_B(f_A(M_A), f_B(M_B))]
= \eta(D)( 1 - e^{-\rho}),
\end{equation}
where $\rho := \rho_A + \rho_B$. 

\begin{defi}			\label{def:NE-2}
The pair of independently randomized strategies $(f_A(M_A), f_B(M_B)) \in \mcC_A \times \mcC_B$
is called a {\em Nash equilibrium} of the game between Alice and Bob
if, for all $g_A \in \mcC_A$ and $g_B \in \mcC_B$, we have
\begin{equation}	\label{eq:NashEqForA}
\mbbE[V_A(f_A(M_A),f_B(M_B))] \ge \mbbE[V_A(g_A, f_B(M_B))],
\end{equation}
and
\begin{equation}		\label{eq:NashEqForB}
\mbbE[V_B (f_A(M_A), f_B(M_B))] \ge \mbbE[V_B(f_A(M_A), g_B)].
\end{equation}
In words, Alice sees no advantage in playing the strategy $g_A$
instead of her randomized strategy $f_A(M_A)$, given that Bob is 
playing his randomized strategy $f_B(M_B)$, and similarly for Bob.
\hfill $\Box$
\end{defi}

\section{Supporting lemmas}		\label{sec:Lemmas}

\subsection{A formula for the value}		\label{subsec:valueformula}

Consider a $2$-player Cox process Hotelling game
$(D, \eta, \rho_A, \rho_B)$ between Alice and Bob, and suppose that
the players play the action pair $(f_A, f_B)$. 
Let $\Phi_A$ and $\Phi_B$ denote the Poisson processes of points of 
Alice and Bob respectively on $D$. Recall that these are independent 
point processes.
Let $\mathbb{P}_A^x$ (resp.
$\mathbb{P}_B^x$) denote the Palm probability
\cite[Chap. 6]{Kallenberg} with respect to $\Phi_A$ 
(resp. $\Phi_B$) at $x$. 
By Slivynak's theorem,
\cite[Lemma 6.14]{Kallenberg}
under the Palm probability with respect to  $\Phi_A$ at $x$, $\Phi_A$ consists of a point at $x$ and
a Poisson process on $D$ of intensity $f_A \eta$, and $\Phi_B$ consists of an independent 
Poisson process on $D$ of intensity $f_B \eta$. The Palm probability with respect to $\Phi_B$ at $x$ 
has a symmetrical description.

From Campbell's formula
\cite[Sec. 6.1]{Kallenberg}
we have
$$V_A(f_A,f_B) = \int_D f_A(x) \mathbb{E}_A^x [\eta(W_\Phi(x))] \eta(dx).$$
Here the expectation is with respect to the law of $\Phi$ under 
the Palm distribution $\mathbb{P}_A^x$. 
From the description of this Palm distribution, we have
\begin{eqnarray*}
 \mathbb{E}_A^x [\eta(W_\Phi(x))] & = &
 \mathbb{E}_A^x [\int_{y\in D} 1_{\{y\in W_\Phi(x)\}} \eta(dy)] \\
&\stackrel{(a)}{=}& \int_{y\in D} \mathbb{P}_A^x( y\in W_\Phi(x)) \eta(dy) \\
&\stackrel{(b)}{=}& \int_{y\in D} e^{-\int_{B(y\to x)} (f_A(u)+f_B(u)) \eta(du)} \eta(dy), 
\end{eqnarray*}
where 
$$B(y\to x) := \{z\in D \mbox{ s.t. } d(z,y)\le d(x,y) \}.$$
The notation $B(y \rightarrow x)$ is supposed to bring to mind
a closed ball centered at $y$ having $x$ at its boundary.

In the chain of equations above, the step (a) comes from 
an application of Fubini's theorem and 
step (b) from Slivnyak's theorem. Hence
\begin{equation}		\label{eq:Valueformula}
V_A(f_A, f_B) =
\int_{x\in D} f_A(x)
\int_{y\in D} e^{-\int_{B(y\to x)} (f_A(u)+f_B(u)) \eta(du)} \eta(dy) \eta(dx).
\end{equation}

\subsection{ $N$-player games}		\label{subsec:N-player}

The point process analysis 
above
and the resulting formula in Equation \eqref{eq:Valueformula} in the
$2$-player case also leads to a formula
in an $N$-player 
Cox process Hotelling game for the value 
$V_i(f_1, \ldots, f_N)$ seen by player $i$
when the pure strategies deployed by the individual players are
$f_1,\ldots, f_N$ respectively. 
For this, let $f_j: D \to \mbbR_+$ be a
nonnegative measurable function on $D$ belonging to  
\begin{equation}		\label{eq:classj}
\mcC_j := \mcC(\rho_j) = \{ f_j:D \to \mbbR_+ \mbox{ s.t. } \int_D f_j(x) \eta(dx) = \rho_j \},
\end{equation}
where $\rho_j > 0$ are fixed for $1 \le j \le N$.
Note that each $\mcC_j$ is a convex subset of $\mcM(D)$, but is not closed in the topology of  weak convergence.
Assume that the individual players play
$f_i \in \mcC_i$ for $1 \le i \le N$.
Let $\Phi_i$ be a Poisson point process on $D$ with intensity measure
$f_i \eta$,
with these processes being mutually independent for $1 \le i \le N$,
and let $\Phi := \sum_{i=1}^N \Phi_i$.
We think of the points of $\Phi_i$ as the points of player $i$, since these
result from the choice of $f_i$, which was made by that
player.
Then the value of player $i$, denoted
$V_i(f_1, \ldots, f_N)$, is given by
$$V_i(f_1, \ldots, f_N) := \mathbb{E} [\sum_{x\in \Phi_i} \eta(W_\Phi(x))].$$
The
expectation is with respect to the joint law of $(\Phi_j, 1 \le j \le n)$,
which are independent.
Here, by definition, a sum over an empty set is 0.

We denote such an $N$-player Cox process Hotelling game by
$(D, \eta, \rho_1, \ldots, \rho_N)$.

We can write a formula for $V_i(f_1, \ldots, f_N)$, based on equation
\eqref{eq:Valueformula}, by thinking of the points
of player $i$ as competing for
space in $D$ with the union of the points of the other players.
Thus we have:
\begin{equation}		\label{eq:value}
V_i(f_1, \ldots, f_N) = \int_{x \in D} f_i (x) \int_{y \in D} 
e^{ - \int_{u \in B(y \rightarrow x)} f(u) \eta(du)} \eta(dy) \eta(dx),
\end{equation}
where $f := \sum_{i=1}^N f_i$.

Note that we must have
\begin{equation}		\label{eq:conservationN}
\sum_{i=1}^N V_i(f_1, \ldots, f_N)  = \eta(D) (1 - e^{- \rho }), 
\end{equation}
where $\rho := \sum_{j=1}^N \rho_j$.
This is because we have $\int_{x \in D} f(x) \eta(dx) = \rho$,
so the law of the total number of points is Poisson with 
mean $\rho$, and so,
by virtue of our assumption that $\eta$ is non-conflicting, the total value of all the players is $\eta(D)$
except on the event that there are no points in $\Phi$, in which case
the total value is $0$.
The formula in Equation \eqref{eq:conservationN} is established in Lemma \ref{lem:totalvalue} below.

\subsection{Nash equilibrium in $N$-player games}
\label{subsec:NEN-defn}

In view of Lemma \ref{lem:AbsContMeas},
any mixed strategy pair in an $N$-player Cox process Hotelling
game $(D, \eta, \rho_1, \ldots, \rho_N)$ can 
be written as 
$(f_1(M_1), \ldots, f_N(M_N))$, where 
$(M_j \in \mcM_j,  1 \le j \le N)$ are independent
random variables representing the randomizations used by 
the individual players in implementing their randomized strategies, and
$f_j(m_j) \in \mcC_j$ is the 
choice of action of player $j$ in case the realization of
her random variable $M_j$ is $m_j$.
The value of player $i$ in such a mixed strategy is 
$\mbbE[ V_i(f_1(M_1), \ldots, f_N(M_N))]$, the expectation being taken
with respect to the joint distribution of $(M_j, 1 \le j \le N)$, which are
independent. 
As a direct consequence 
of Equation \eqref{eq:conservationN}, we have
\begin{equation}		\label{eq:randconserveN}
\sum_{i=1}^N \mbbE[ V_i(f_1(M_1), \ldots, f_N(M_N))]
= \eta(D)|( 1 - e^{-\rho}),
\end{equation}
where $\rho :=  \sum_{i=1}^N \rho_i$. 

\begin{defi}			\label{def:NE-N}
The vector of independently randomized strategies 
\[
(f_1(M_1), \ldots, f_N(M_N)) \in \mcC_1 \times  \ldots \times \mcC_N
\]
is called a {\em Nash equilibrium} of the 
$N$-player game
if, for all $g_j \in \mcC_j$, $1 \le j \le N$, we
have, for all $1 \le i \le N$, 
\begin{equation}		\label{eq:NashEqForPlayeri}
\mbbE[V_i(f_1(M_1), \ldots, f_N(M_N))] \ge 
\mbbE[V_i(g_i, (f_j(M_j), j \neq i))].
\end{equation}
In words, each player $i$ perceives no advantage in playing the 
strategy $g_i$ instead of its randomized strategy $f_i(M_i)$, 
given that the other players, i.e. the players $j \neq i$, are playing the 
individually randomized strategies $(f_j(M_j), j \neq i)$.
\hfill  $\Box$
\end{defi}

\subsection{A conservation law}		\label{subsec:constantsum}

In order to 
establish the conservation law in Equation \eqref{eq:conservationN}
and its special case in Equation \eqref{eq:conservation2}, it suffices to
demonstrate that for all $\rho > 0$ and all 
$f \in \mcC(\rho)$, we have equation
\eqref{eq:tautology}
below.
This is because the 
desired conservation laws would then follow immediately from 
equation
\eqref{eq:value} 
and its special case 
in Equation \eqref{eq:Valueformula},
respectively, by summing these formulas over the individual players.
We establish 
\eqref{eq:tautology}
formally in the following lemma.

\begin{lem}
\label{lem:totalvalue}
For any $\rho > 0$
and $ f \in \mcC(\rho)$, we have:
\begin{equation}		\label{eq:tautology}
\int_{x\in D} f(x)
\int_{y\in D} e^{-\int_{B(y\to x)} f(u) \eta(du)} \eta(dy) \eta(dx)= \eta(D) \left(1-e^{-\rho }\right).
\end{equation}
\end{lem}

\begin{proof}
We give two proofs of this formula.
The first one is probabilistic. Let $\Phi$ denote a Poisson process on
$D$ with intensity measure $f \eta$.
By Campbell's formula and Slivnyak's theorem
\cite{Kallenberg},
the integral on the left hand side of Equation \eqref{eq:tautology} is just
$$ \mathbb{E} [\sum_{x \in \Phi} \eta(W_\Phi(x))],$$
where $W_\Phi(x)$ denotes the open Voronoi cell of $x \in \Phi$ 
with respect to the Poisson process $\Phi$.
If $\Phi(D)=0$, which happens with probability $e^{-\rho}$,
the sum is 0. On the complementary event the sum is $\rho(D)$, since, by virtue of the assumption that $\eta$ is non-conflicting, the collection
of sets $(W_\Phi(x), x \in \Phi)$ form a partition of $D$ up to a set of $\eta$-measure 0.

The second proof is analytical.
We have, for each $y \in D$,
\[
\int_{x \in D} f(x) e^{ - \int_{u \in B(y \rightarrow x)} f(u) \eta(du)}  \eta(dx) = 
1 - e^{- \rho },
\]
since, when $\eta$ is non-conflicting, the integral on the left hand side of the preceding equation is
just the probability that a nonhomogeneous
Poisson process with intensity function $f(x)$ with respect to $\eta$ has at least one
point in $D$ (to see this, think of moving out from $y$ by balls of
increasing radius till one covers all of $D$). 
Thus we have
\[
\int_{y \in D} \int_{x \in D} f(x)
e^{ - \int_{u \in B(y \rightarrow x)} f(u) \eta(du)} \eta(dy) \eta(dx) = 
\eta(D) (1 - e^{- \rho }).
\]
But the left hand side of the preceding equation equals the left hand
side of Equation \eqref{eq:tautology}, by an application of Fubini's theorem.

\end{proof}

\subsection{Constant strategies}		\label{subsec:constant}

Consider a $2$-player Cox process Hotelling game
$(D, \eta, \rho_A, \rho_B)$.
Write $\barrho_A$ for the
strategy of Alice where she chooses $f_A$ to be the constant
$\frac{\rho_A}{\eta(D)}$, and similarly write 
$\barrho_B$ for the
strategy of Bob where he chooses $f_B$ to be the constant
$\frac{\rho_B}{\eta(D)}$. Note that $\barrho_A \in \mcC_A$
and $\barrho_B \in \mcC_B$. 

The following useful lemma gives an explicit formula for the value 
obtained by Alice and that obtained by Bob in the $2$-player game 
$(D, \eta, \rho_A, \rho_B)$ when the players
play the constant strategies $\barrho_A$ and $\barrho_B$ respectively.
Note that, while the strategies are constant, the resulting intensities are
$\frac{\rho_A}{\eta(D)} \eta$ and $\frac{\rho_B}{\eta(d)} \eta$ for Alice and Bob respectively.

\begin{lem}
\label{lem:whenconstant2}
Consider a $2$-player Cox process Hotelling game $(D, \eta, \rho_A, \rho_B)$.
Then $V_A(\barrho_A, \barrho_B) = \frac{\rho_A}{\rho} \eta(D) (1 - e^{- \rho})$ and 
$V_B(\barrho_A, \barrho_B) = \frac{\rho_B}{\rho} \eta(D) (1 - e^{- \rho})$,
where $\rho := \rho_A + \rho_B$.
\end{lem}

\bpf
We have 
\begin{eqnarray*}
V_A(\barrho_A, \barrho_B) &\stackrel{(a)}{=}& 
\int_{x \in D} \frac{\rho_A}{\eta(D)} \int_{y \in D} 
e^{ - \frac{\rho}{\eta(D)} \eta(B(y \rightarrow x))} \eta(dy) \eta(dx)\\
&=& \frac{\rho_A}{\rho} 
\int_{x \in D} \int_{y \in D} 
\frac{\rho}{\eta(D)} e^{ - \frac{\rho}{\eta(D)} \eta(B(y \rightarrow x))} \eta(dy) \eta(dx)\\
&\stackrel{(b)}{=}& \frac{\rho_A}{\rho} 
\int_{y \in D} \int_{x \in D} 
\frac{\rho}{\eta(D)} e^{ - \frac{\rho}{\eta(D)} \eta(B(y \rightarrow x))} \eta(dx) \eta(dy)\\
&\stackrel{(c)}{=}& \frac{\rho_A}{\rho} \eta(D) (1 - e^{- \rho}).
\end{eqnarray*}
Here step (a) is from equation
\eqref{eq:Valueformula},
step (b) is an application of Fubini's theorem,
and 
step (c) comes from Lemma \ref{lem:totalvalue}.
The formula for $V_B(\barrho_A, \barrho_B)$ results from
interchanging the roles of Alice and Bob in this calculation.
\epf

In an $N$-player Cox process Hotelling game
$(D, \eta, \rho_1, \ldots, \rho_M)$, write 
$\barrho_i$ for the
strategy of player $i$, $1 \le i \le N$, 
where she chooses $f_i$ to be the constant
$\frac{\rho_i}{\eta(D)}$. We then have the following analog of 
Lemma \ref{lem:whenconstant2}.

\begin{lem}
\label{lem:whenconstantN}
Consider an $N$-player Cox process Hotelling game 
$(D, \eta, \rho_1, \ldots, \rho_N)$.
Then, for all $1 \le i \le N$, we have
\[
V_i(\barrho_1, \ldots, \barrho_N) = \frac{\rho_i}{\rho} \eta(D) (1 - e^{- \rho}),
\]
where $\rho := \sum_{i=1}^N \rho_i$.
\end{lem}

\bpf
The proof is similar to that of Lemma \ref{lem:whenconstant2},
when one starts with Equation \eqref{eq:value} instead of 
Equation \eqref{eq:Valueformula} and makes the obvious modifications.
\epf

\subsection{Concavity of the value}		\label{subsec:concavity}

To close this section, we record a convexity property that will
play a key role in establishing the main claims of this paper.

\begin{lem}
\label{lem:convexity2}
Consider a $2$-player Cox process Hotelling game $(D, \eta, \rho_A, \rho_B)$.
Fix $f_B \in \mcC_B$
and $x \in D$. Then the mapping 
\[
f_A \mapsto \int_{y \in D} 
e^{ - \int_{B(y \rightarrow x)} (f_A(u) + f_B(u)) \eta(du) } \eta(dy)
\]
is strictly convex on $\mcC_A$, which, we recall, is a convex set.

As a consequence, $f_A \mapsto V_B(f_A, f_B)$,
 for fixed $f_B \in \mcC_B$, is strictly convex 
on $\mcC_A$, and hence $f_A \mapsto V_A(f_A, f_B)$,
 for fixed $f_B \in \mcC_B$, is strictly concave 
on $\mcC_A$.
\end{lem}

\bpf
Let $f_A, f'_A \in \mcC_A$
and $\theta \in [0,1]$.
Then we have
\begin{eqnarray*}
&&~ \int_{y\in D} e^{-\int_{B(y\to x)} (\theta f_A(u) + (1-\theta) f'_A(u) + f_B(u)) \eta(du)} \eta(dy)\\
&&~~~~  \le 
\theta
\int_{y\in D} e^{-\int_{B(y\to x)} (f_A(u)+ f_B(u)) \eta(du)} \eta(dy)\\
&&~~~~ +
(1-\theta)
\int_{y\in D} e^{-\int_{B(y\to x)} (f'_A(u)+ f_B(u)) \eta(du)} \eta(dy),
\end{eqnarray*}
from the convexity of the exponential function, and this inequality is
strict if $\theta \notin \{0,1\}$ and $f_A \neq f'_A$.
This proves the first claim of the lemma.

From Equation \eqref{eq:Valueformula}, we have
\begin{equation}		\label{eq:convexfunctional}
V_B(f_A, f_B) = \int_{x \in D} f_B(x) \int_{y \in D} 
e^{ - \int_{B(y \rightarrow x)} (f_A(u) + f_B(u)) \eta(du) } \eta(dy) \eta(dx).
\end{equation}
Since, for fixed $f_B \in \mcC_B$,
the inner integral is strictly convex on $\mcC_A$ for 
each $x \in D$, the overall integral is also strictly convex on $\mcC_A$,
which establishes the second claim of the lemma.

Finally, from the conservation rule in Equation \eqref{eq:conservation2},
we have $V_A(f_A, f_B) = \eta(D)(1 - e^{-\rho}) - V_B(f_A,f_B)$,
where $\rho := \rho_A + \rho_B$. From the second claim of the lemma,
the third claim now follows immediately.
\epf

The main claim of Lemma \ref{lem:convexity2} is the third one
about the strict concavity of the value function. This also holds
for $N$-player Cox process Hotelling games. We state this claim formally in the 
following lemma.

\begin{lem}
\label{lem:convexityN}
Consider an $N$-player Cox process Hotelling game 
$(D, \eta, \rho_1, \ldots, \rho_N)$.
Recall that the sets of pure actions, $\mcC_i$ for player $i$,
$1 \le i \le N$, as defined in Equation \eqref{eq:classj}, 
are convex subsets of $\mcM(D)$.

For each $1 \le i \le N$, the map 
$f_i \mapsto V_i(f_1, \ldots, f_N)$,
for fixed $f_j \in \mcC_j$, $j \neq i$, is strictly concave 
on $\mcC_i$.
\end{lem}

\bpf
This is a direct consqeuence of the third claim of Lemma 
\ref{lem:convexity2}, once one observes that value of player $i$
when it plays $f_i$ in response to the actions $f_j$ of the players
$j \neq i$ in the $N$-player game is that same as its value when 
it plays $f_i$ in response to the pure action 
$\sum_{j \neq i} f_j$ of the opposing player in a $2$-player game
where the intensity constraint of the opposing player is
$\sum_{j \neq i} \rho_j$.
\epf

\section{Invariance under a transitive group action}		\label{sec:Homogeneous}

In this section, we consider the case where
 $D$ is compact,
admits
a transitive group of 
metric-preserving automorphisms,
and 
$\eta$ is 
invariant under this group of automorphisms.
This scenario covers several interesting
concrete cases, such as the metric tori derived from
lattice fundamental regions in $\mbbR^d$, with the metric of 
$\mbbR^d$ and Lebesgue measure; spheres of a fixed radius
with the associated uniform measure, which is 
invariant under the rigid 
rotations; etc.

\subsection{Exploiting concavity of the value}		\label{subsec:smoothing}

The following lemma, which it suffices to state 
in the $2$-player case,
is the key technical result driving the 
game-theoretic results in this section.

\begin{lem}		\label{lem:main}
Consider a $2$-player Cox process Hotelling game $(D, \eta, \rho_A, \rho_B)$ 
where $D$ is compact, admits a transitive group of metric-preserving automorphisms,
and $\eta$ is 
invariant under this group of automorphisms.
Then, for any strategy $f_A \in \mcC_A$ of Alice,
we have 
\[
V_A (f_A, \barrho_B) \le V_A (\barrho_A, \barrho_B) = \frac{\rho_A}{\rho}
 \eta(D) (1 - e^{- \rho }),
\]
where $\rho := \rho_A + \rho_B$. 
Further, we have the strict inequality
\[
V_A (f_A, \barrho_B) < V_A (\barrho_A, \barrho_B),
\]
except in the case $f_A = \barrho_A$.
\end{lem}

\bpf
From Lemma \ref{lem:whenconstant2}, we have
$V_A(\barrho_A, \barrho_B) = \frac{\rho_A}{\rho} \eta(D) (1 - e^{- \rho})$,
which is one of the claims of this lemma.
From Lemma \ref{lem:convexity2} for the 
choice $f_B = \barrho_B$, we conclude that 
$V_A(f_A, \barrho_B)$ is strictly concave on $\mcC_A$.
We now use this 
to conclude that 
$V_A(f_A, \barrho_B)$  
is uniquely maximized over $\mcC_A$
by the choice $f_A = \barrho_A$. Indeed, if $f_A \in \mcC_A$, $f_A \neq \barrho_A$, then we can find a translate $f'_A$ of $f_A$ such that
$f_A \neq f'_A$. We have 
$V_A(f_A, \barrho_B) = V_A(f'_A,\barrho_B)$, because 
$f_A$ and $f'_A$ are translates of each other. However, 
since $f_A \neq f'_A$, we have
\[
V_A(\frac{1}{2}(f_A + f'_A), \barrho_B) > \frac{1}{2} V_A(f_A, \barrho_B) + \frac{1}{2} V_A(f'_A, \barrho_B) = V_A(f_A, \barrho_B).
\]
Hence, $f_A$ cannot be the maximizer of $V_A(f_A, \barrho_B)$
over $\mcC_A$ unless it is translation
invariant, i.e. unless it equals $\barrho_A$.
This concludes the proof of the lemma.
\epf

\subsection{Nash equilibrium structure for $2$-player games}
\label{subsec:NE2-invariant}


We are now in a position to determine the Nash equilibrium 
structure of $2$-player Cox process Hotelling games 
$(D, \eta, \rho_A, \rho_B)$ in the context of this section.


\begin{thm}			\label{thm:unique2}
Consider a $2$-player Cox process Hotelling game
$(D, \eta, \rho_A, \rho_B)$ between two 
players, Alice and Bob,
where $D$ is compact,
admits a transitive group of metric-preserving automorphisms,
and $\eta$ is
invariant under this group of automorphisms.
Then $(\barrho_A, \barrho_B)$ is the unique Nash equilibrium for the game.

\end{thm}

\bpf

We need to show that if $(f_A(M_A), f_B(M_B))$ 
is a Nash equilibrium for the 
game, then $f_A(M_A) = \barrho_A$ and 
$f_B(M_B) = \barrho_B$
with probability $1$.

Suppose first that $f_B(M_B) = \barrho_B$ with probability $1$. 
If $\mbbP(f_A(M_A) \neq \barrho_A) > 0$ then,
by Lemma \ref{lem:main}, we have 
\[
\mbbE[V_A(f_A(M_A),f_B(M_B))] = \mbbE[V_A(f_A(M_A), \barrho_B)] 
< V_A(\barrho_A, \barrho_B).
\]
On the other hand, since $f_A(M_A)$ is a best response by Alice to the 
strategy $f_B(M_B)$ of Bob, we have
\[
\mbbE[V_A(f_A(M_A),f_B(M_B))] \ge \mbbE[V_A(\barrho_A,f_B(M_B))]
= V_A(\barrho_A, \barrho_B).
\]
This contradiction establishes that if $f_B(M_B) = \barrho_B$ with probability $1$ then we must have $f_A(M_A) = \barrho_A$ with 
probability $1$.
A similar argument works to show that if $f_A(M_A) = \barrho_A$ with 
probability $1$ then we must have $f_B(M_B) = \barrho_B$ with probability $1$. 

Thus, it remains to handle the case where 
we have both $\mbbP(f_A(M_A) \neq \barrho_A) > 0$ and 
$\mbbP(f_B(M_B) \neq \barrho_B) > 0$.
In this case, since $f_A(M_A)$ is a best response by Alice to the 
strategy $f_B(M_B)$ of Bob, we have
\begin{equation}		\label{eq:abound}
\mbbE[V_A(f_A(M_A),f_B(M_B))] \ge \mbbE[V_A(\barrho_A,f_B(M_B))],
\end{equation}
and, since $f_B(M_B)$ is a best response by Bob to the 
strategy $f_A(M_A)$ of Alice, we have
\begin{equation}		\label{eq:bbound}
\mbbE[V_B(f_A(M_A),f_B(M_B))] \ge \mbbE[V_B(f_A(M_A), \barrho_B)].
\end{equation}
Now, since $\mbbP(f_B(M_B) \neq \barrho_B) > 0$, by Lemma \ref{lem:main}
we have 
\[
\mbbE[V_B(\barrho_A,f_B(M_B))] < V_B(\barrho_A, \barrho_B),
\]
so that, by Equation \eqref{eq:randconserve2}, we have 
\[
\mbbE[V_A(\barrho_A,f_B(M_B))] > V_A(\barrho_A, \barrho_B).
\]
Combining this with Equation \eqref{eq:abound}, we get
\[
\mbbE[V_A(f_A(M_A),f_B(M_B))]  > V_A(\barrho_A,\barrho_B).
\]
Similar reasoning, based on Equation \eqref{eq:bbound}, gives
\[
\mbbE[V_B(f_A(M_A),f_B(M_B))]  > V_B(\barrho_A,\barrho_B).
\]
However, putting these inequalities together contradicts the conservation law in equation
\eqref{eq:randconserve2}. This completes the proof of the theorem.

\epf

\subsection{Nash equilibrium structure for $N$-player games}
\label{subsec:NEN-invariant}

Theorem \ref{thm:unique2} is actually a special case of a
uniqueness theorem for Nash equilibria in the general $N$-player case.
The proof in the $N$-player case also
depends on a peculiar feature of Cox process Hotelling
games, which is that a player faced with the strategies of the other players,
i.e. their individual choices of likelihood functions with respect to the underlying measure $\eta$ which result in their individual intensities,
receives the same value as she would in a $2$-player game where she is
faced with a single player playing a likelihood with respect to the underlying measure $\eta$ that results in an intensity equal to the sum of the intensities corresponding
to the strategies of the other players. With this in mind, we turn now to the
$N$-player case.

\begin{thm}			\label{thm:uniqueN}
Consider an $N$-player Cox process Hotelling game
$(D, \eta, \rho_1, \ldots, \rho_N)$
where $D$ is compact,
admits a transitive group of metric-preserving automorphisms,
and $\eta$ is
invariant under this group of automorphisms.
Let
$\barrho_j \in \mcC_j$ denote the constant function 
$\frac{\rho_j}{\eta(D)}$, which results in the constant intensity
$\frac{\rho_j}{\eta(D)} \eta$ for player $j$.

Then $(\barrho_1, \ldots, \barrho_N)$ is the unique Nash equilibrium for this game.

\end{thm}

\bpf

We need to show that if $(f_1(M_1), \ldots, f_N(M_N))$ is a 
Nash equilibrium, then $\mbbP(f_j(M_j) = \barrho_j) = 1$ for all
$1 \le j \le N$. To do this, suppose first, after reindexing if needed,
that we have $\mbbP(f_1(M_1) \neq \barrho_1) > 0$ and 
$\mbbP(\sum_{j=2}^N f_j(M_j) = \sum_{j=2}^N \barrho_j) = 1$.
Then, because $f_1(M_1)$ is a best reaction of player $1$ 
to the individually randomized strategies $(f_j(M_j), 2 \le j \le N)$
of the other players, we have
\begin{equation}		\label{eq:constass-1}
\mbbE[V_1(f_1(M_1), f_2(M_2), \ldots, f_N(M_N))] \ge 
\mbbE[V_1(\barrho_1, f_2(M_2), \ldots, f_N(M_N))].
\end{equation}
On the other hand, by Lemma \ref{lem:main}, we have 
\begin{equation}		\label{eq:constass-2}
\mbbE[V_1(f_1(M_1), f_2(M_2), \ldots, f_N(M_N))] < 
\mbbE[V_1(\barrho_1, f_2(M_2), \ldots, f_N(M_N))].
\end{equation}
To see this, observe that 
$V_1(f_1(M_1), f_2(M_2), \ldots, f_N(M_N))$ is the same
as the value of player $1$ in the $2$-player game in which she plays the randomized
strategy $f_1(M_1)$ against a single opponent playing the
strategy $\sum_{j=2}^N f_j(M_j)$, which we have assumed equals
the constant $\sum_{j=2}^N \barrho_j$ with probability $1$, and 
also $V_1(\barrho_1, f_2(M_2), \ldots, f_N(M_N))$ is the same
as the value of player $1$ in the $2$-player game in which she plays 
the constant strategy $\barrho_1$ against a single opponent playing the
strategy $\sum_{j=2}^N f_j(M_j)$, which we have assumed equals
the constant $\sum_{j=2}^N \barrho_j$ with probability $1$, and 
so the scenario of Lemma \ref{lem:main} applies to allow us to 
compare these two values. Note that equations
\eqref{eq:constass-1} and \eqref{eq:constass-2} contradict each other.
Thus, we can conclude that if $(f_1(M_1), \ldots, f_N(M_N))$ is a 
Nash equilibrium then for every player $1 \le i \le N$ for which
$\mbbP(f_i(M_i) \neq \barrho_i) > 0$, we must also have 
$\mbbP(\sum_{j \neq i} f_j(M_j) \neq \sum_{j \neq i} \barrho_j) > 0$.

Suppose now, after reindexing if necessary, that 
$\mbbP(f_1(M_1) \neq \barrho_1) > 0$. We have established that
we must also have
$\mbbP(\sum_{j=2}^N f_j(M_j) \neq \sum_{j=2}^N \barrho_j) > 0$.
Since $f_1(M_1)$ is a best reaction of player $1$ 
to the individually randomized strategies $(f_j(M_j), 2 \le j \le N)$
of the other players, the inequality in Equation \eqref{eq:constass-1} holds. 
Because $\mbbP(\sum_{j=2}^N f_j(M_j) \neq \sum_{j=2}^N \barrho_j) > 0$,
from Lemma \ref{lem:main} also have 
\begin{equation}		\label{eq:nonconst-1}
\sum_{j=2}^N \mbbE[ V_j(\barrho_1, f_2(M_2), \ldots, f_N(M_N))]
< \sum_{j=2}^N V_j(\barrho_1, \barrho_2, \ldots, \barrho_N).
\end{equation}
To see this, note that the sum of the values of the other players 
$2 \le j \le N$, when player $1$ plays the constant strategy
$\barrho_1$, is the same as the value of a single player playing the 
strategy $\sum_{j=2}^N f_j(M_j)$ against player $1$ playing the 
constant strategy $\barrho_1$ in a $2$-player game between player $1$ and this single player, where the overall intensity constraint of player $1$ continues to be $\rho_1$ and that of this single player is 
$\sum_{j=2}^N \rho_j$. Since $\mbbP(\sum_{j=2}^N f_j(M_j) \neq \sum_{j=2}^N \barrho_j) > 0$, Lemma \ref{lem:main} allows us to 
conclude that this value is strictly less the value this single player
would get by playing the constant strategy
$\sum_{j=2}^N \barrho_j$ against player $1$, who is playing the constant
strategy $\barrho_1$. But this is equal to the sum of the values of 
individual players $2 \le j \le N$ in the given $N$-player game 
when they individually play the 
constant strategies $(\barrho_j, 2 \le j \le N)$ respectively, and player
$1$ is playing the constant strategy $\barrho_1$.

Now, in view of the conservation law in Equation \eqref{eq:randconserveN}, we can conclude from
Equation \eqref{eq:nonconst-1} that 
\[
\mbbE[V_1(\rho_1, f_2(M_2), \ldots, f_N(M_N))]
> V_1(\barrho_1, \barrho_2, \ldots, \barrho_N).
\]
Thus, so far, what we have concluded is that
if $(f_1(M_1), \ldots, f_N(M_N))$ is a 
Nash equilibrium, then,
every $1 \le i \le N$ such that
$\mbbP(f_i(M_i) \neq \barrho_i) > 0$, we must have
\begin{equation}		\label{eq:nonconst-2}
\mbbE[V_i(\barrho_i, (f_j(M_j), j \neq i))]
> V_i(\barrho_1, \barrho_2, \ldots, \barrho_N).
\end{equation}

Finally, suppose that 
$\mbbP(f_i(M_i) = \barrho_i) = 1$ for some $1 \le i \le N$.
Then we must
\begin{equation}		\label{eq:easycase}
\mbbE[V_i(\barrho_i, (f_j(M_j), j \neq i))]
\ge V_i(\barrho_1, \barrho_2, \ldots, \barrho_N).
\end{equation}
To see this note that the
$\sum_{k \neq i} \mbbE[V_k(\barrho_i, (f_j(M_j), j \neq i))]$
is the same as the value of a single player who plays the strategy
$\sum_{j \neq i} f_j(M_j)$ in the $2$-player game against player 
$i$ playing the constant strategy $\barrho_i$ and, by Lemma \ref{lem:main}
this is no bigger that the value this single player would get if she played
the constant strategy $\sum_{j \neq i} \barrho_j$, but this value is 
the same as $\sum_{j \neq i} V_j(\barrho_1, \barrho_2, \ldots, \barrho_N)$.
To conclude Equation \eqref{eq:easycase} from this logic, apply the
conservation rule in Equation \eqref{eq:randconserveN}.

The inequalities in equations \eqref{eq:nonconst-2} and \eqref{eq:easycase}
together result in a contradiction of Equation \eqref{eq:randconserveN} 
unless we have $\mbbP(f_i(M_i) = \barrho_i) = 1$ for all $1 \le i \le N$.
This concludes the proof of the theorem.

\epf

\section{General results}		\label{sec:General}



In this section we discuss 
the structure of
Nash equilibria in a general $N$-player Cox process Hotelling game
$(D, \eta, \rho_1,\ldots, \rho_N)$ without the 
group-theoretic assumptions
of Section \ref{sec:Homogeneous}.

\subsection{Structure of the Nash equilibria, assuming one exists}
\label{subsec:GenStruct}


Leaving aside for the moment the question of existence of Nash equilibria,
the strict concavity of the value function of a player for fixed choices of the pure actions of the other players, which was established in 
Lemmas \ref{lem:convexity2} and \ref{lem:convexityN},  ensures that
any Nash equilibrium that exists must be pure.
We discuss this first for the $2$-player case.

\begin{thm} 		 \label{thm:purity2}
Consider a $2$-player Cox process Hotelling game
$(D, \eta, \rho_A, \rho_B)$ between the
players Alice and Bob. 
Suppose $(f_A(M_A), f_B(M_B)) \in \mcC_A \times \mcC_B$
is a Nash equilibrium of the game, as defined in 
Definition \ref{def:NE-2}. Then there exist
$g_A \in \mcC_A$ and $g_B \in \mcC_B$ such that 
\[
\mbbP((f_A(M_A), f_B(M_B)= (g_A, g_B)) = 1.
\]
\end{thm}

\bpf
We write
\begin{eqnarray*}
\mbbE[V_A(f_A(M_A),f_B(M_B))] 
&=& \mbbE[ \mbbE[ V_A(f_A(M_A),f_B(M_B))|M_B]]\\
&\stackrel{(a)}{\le}& \mbbE[ \mbbE[V_A( \mbbE[f_A(M_A)|M_B], f_B(M_B)|M_B]]\\
&\stackrel{(b)}{=}& \mbbE[ \mbbE[ V_A(\mbbE[f_A(M_A)], f_B(M_B))|M_B]]\\
&=& \mbbE[ V_A( \mbbE[f_A(M_A)], f_B(M_B)].
\end{eqnarray*}
Here step (a) comes from the concavity property of the value function
established in Lemma \ref{lem:convexity2} and step (b) comes from
independence of $M_A$ and $M_B$. Since 
$(f_A(M_A), f_B(M_B))$ is a Nash equilibrium pair, we see
from Definition \ref{def:NE-2} that the inequality in the 
chain of equations above must be an equality. But then, by the 
strict concavity property of the value function
established in Lemma \ref{lem:convexity2}, it follows that 
$\mbbP(f_A(M_A) = \mbbE[f_A(M_A)]) = 1$. A similar argument interchanging
the roles of Alice and Bob completes the proof, with 
$g_A$ being $\mbbE[f_A(M_A)]$ and $g_B$ being 
$\mbbE[f_B(M_B)]$ in the notation of the statement of the lemma.
\epf

The analog of Theorem \ref{thm:purity2} also holds 
in the $N$-player case.

\begin{thm} 		 \label{thm:purityN}
Consider an $N$-player Cox process Hotelling game
$(D, \eta, \rho_1, \ldots, \rho_N)$.
Suppose $(f_1(M_1), \ldots, f_N(M_N)) \in \mcC_1 \times
\ldots \times  \mcC_N$
is a Nash equilibrium of the game, as defined in Definition
\ref{def:NE-N}. Then there exist
$g_i \in \mcC_i$, $1 \le i \le N$, such that 
\[
\mbbP((f_1(M_1), \ldots, f_N(M_N) = (g_1, \ldots, g_N)) = 1.
\]
\end{thm}

\bpf
The proof is similar to that of Theorem \ref{thm:purity2},
with the obvious modifications. The key observation, if one 
wants to base the proof of Lemma \ref{lem:convexity2}, is that 
the value of player $i$ when she plays the randomized strategy
$f_i(M_i)$ in response to the randomized strategies 
$f_j(M_j)$, $j \neq i$ of the other players in the $N$-player
game is the same as her value when she plays the randomized 
strategy $f_i(M_i)$ in response to the randomized strategy
$\sum_{j \neq i} f_j(M_j)$ of the opposing player in a 
$2$-player game where the opposing player has the intensity
constraint $\sum_{j \neq i} \rho_j$. 
This observation then leads to the conclusion that 
$P( f_i(M_i) = E[f_i(M_i)]) = 1$, by following the lines of the
proof of Theorem \ref{thm:purity2}, and since this holds for all
$1 \le i \le N$, this completes the proof.

Alternately, one can 
write out the obvious analog of the sequence of equations
in the proof of Theorem \ref{thm:purity2} 
by conditioning on $(M_j, j \neq i)$, for each $1 \le i \le N$,
and base the proof
on the stricty concavity property for the $N$-player game proved
in Lemma \ref{lem:convexityN}.
\epf

In fact the strict concavity property of the value function of a player that was established in Lemmas 
\ref{lem:convexity2} and \ref{lem:convexityN}, together with the constant sum nature of 
the game, ensures that, if a Nash equilibrium exists, then it is
not only pure, but is unique. We state this first in the $2$-player
case.

\begin{thm} 		 \label{thm:pureandunique2}
Consider a $2$-player Cox process Hotelling game
$(D, \eta, \rho_A, \rho_B)$ between the
players Alice and Bob. 
Suppose $(f_A, f_B) \in \mcC_A \times \mcC_B$
and $(g_A, g_B) \in \mcC_A \times \mcC_B$ are pure
Nash equilibria of the game.
Then $f_A = g_A$ and $f_B = g_B$.
\end{thm}

\bpf
We will first show that 
\begin{equation}	\label{eq:valueexists}
V_A(g_A,g_B) = V_A(f_A,f_B).
\end{equation}
The procedure to do this is standard in the theory of constant sum games,
but is reproduced here for convenience. 
To verify Equation \eqref{eq:valueexists},
 note that $V_A(g_A, f_B) \le V_A(f_A, f_B)$, because
$(f_A, f_B)$ is a Nash equilibrium, see Equation \eqref{eq:NashEqForA}.
But then, by the constant sum nature of the game, see Equation 
\eqref{eq:randconserve2}, we have 
$V_B(g_A,f_B) \ge V_B(f_A,f_B)$. However, since 
$(g_A, g_B)$ is a Nash equilibrium, we have
$V_B(g_A,g_B) \ge V_B(g_A,f_B)$, so we conclude that 
$V_B(g_A,g_B) \ge V_B(f_A,f_B)$. Interchanging the roles of 
$(f_A, f_B)$ and $(g_A, g_B)$ then gives 
$V_B(f_A,f_B) \ge V_B(g_A,g_B)$, which establishes
Equation \eqref{eq:valueexists}.

We next show that 
\begin{equation}		\label{eq:crisscross}
V_A(g_A, f_B) = V_A(f_A, f_B).
\end{equation}
We have $V_B(g_A,g_B) \ge V_B(g_A,f_B)$ because 
$(g_A, g_B)$ is a Nash equilibrium. Hence, 
 by the constant sum nature of the game, see Equation 
\eqref{eq:randconserve2}, we have 
$V_A(g_A,g_B) \le V_A(g_A,f_B)$. In view of 
Equation \eqref{eq:valueexists}, this gives 
$V_A(f_A,f_B) \le V_A(g_A,f_B)$, but since 
$(f_A, f_B)$ is a Nash equilibrium, this can only hold with 
equality, i.e. Equation \eqref{eq:crisscross} holds.

Now, since $(f_A, f_B)$ is a Nash equilibrium we know that
$f_A$ is a best response of Alice to the pure strategy $f_B$ of
Bob. Thus, Equation \eqref{eq:crisscross} tells us that
$g_A$ is also a best response of Alice in response to $f_B$.
The strict concavity property of the value function of Alice proved
in Lemma \ref{lem:convexity2} shows that this is only possible
if $g_A = f_A$. Interchanging the roles of Alice and Bob, we 
conclude that we must also have $g_B = f_B$. This concludes
the proof of the lemma.
\epf

Assume Nash equilibria exist.
Beyond there being a unique pure Nash equilibrium,
the special structure of the game allows
us to say more about the form of this unique Nash equilibrium.
We state the result first in the $2$-player case.

\begin{thm} 		 \label{thm:proportional2}
Consider a $2$-player Cox process Hotelling game
$(D, \eta, \rho_A, \rho_B)$ between the
players Alice and Bob. 
Suppose $(f_A, f_B) \in \mcC_A \times \mcC_B$
is a pure
Nash equilibrium of the game.
Then $f_A = \rho_A f$ and $f_B = \rho_B f$,
where $f := f_A + f_B$.
\end{thm}

\bpf
Suppose Alice reacts to $f_B$ by playing $\frac{\rho_A}{\rho_B} f_B$.
Let $\rho := \rho_A + \rho_B$. 
We have 
\begin{eqnarray*}
V_A(\frac{\rho_A}{\rho_B} f_B, f_B) &\stackrel{(a)}{=}&
\frac{\rho_A}{\rho_B} \int_{x\in D} f_B(x)
\int_{y\in D} e^{-\int_{B(y\to x)} \frac{\rho}{\rho_B} f_B(u) \eta(du)} \eta(dy) \eta(dx)\\
&=& \frac{\rho_A}{\rho} 
\int_{x\in D} \frac{\rho}{\rho_B}  f_B(x)
\int_{y\in D} e^{-\int_{B(y\to x)} \frac{\rho}{\rho_B} f_B(u) \eta(du)} \eta(dy) \eta(dx)\\
&\stackrel{(b)}{=}& \frac{\rho_A}{\rho} \eta(D) (1 - e^{-\rho}),
\end{eqnarray*}
where step (a) is from Equation \eqref{eq:Valueformula} 
and step (b) is from Equation \eqref{eq:tautology}. 
Since $(f_A, f_B)$ is a Nash equilibrium, it follows that 
\[
V_A(f_A,f_B) \ge \frac{\rho_A}{\rho} \eta(D) (1 - e^{-\rho}).
\]
Interchanging the roles of Alice and Bob gives
\[
V_B(f_A,f_B) \ge \frac{\rho_B}{\rho} \eta(D) (1 - e^{-\rho}).
\]
In view of the constant sum property in Equation 
\eqref{eq:randconserve2}, it then follows that we have equality in 
both these inequalities, i.e.
\[
V_A(f_A,f_B) = \frac{\rho_A}{\rho} \eta(D) (1 - e^{-\rho}),
\]
and
\[
V_B(f_A,f_B) = \frac{\rho_B}{\rho} \eta(D) (1 - e^{-\rho}).
\]
But then we have $V_A(f_A,f_B) = V_A(\frac{\rho_A}{\rho_B} f_B, f_B)$, so the strict concavity property of the value of Alice 
proved in Lemma \ref{lem:convexity2} together with the 
assumption that  $(f_A, f_B)$ is a Nash equilibrium
implies that $f_A = \frac{\rho_A}{\rho_B} f_B$,
which suffices to establish the claim (one can also go through this
argument again interchanging the roles of Alice and Bob, but it is
not necessary).
\epf

The analog of Theorem \ref{thm:pureandunique2} 
and the stronger statement in Theorem \ref{thm:proportional2} 
also hold in the $N$-player case. We state these claims together.

\begin{thm} 		 \label{thm:proportionalN}
Consider an $N$-player Cox process Hotelling game
$(D, \eta, \rho_1, \ldots, \rho_N)$.
Suppose 
$(f_1, \ldots, f_N) \in \mcC_1 \times  \ldots \times \mcC_N$
and
$(g_1, \ldots, g_N) \in \mcC_1 \times  \ldots \times \mcC_N$
are pure
Nash equilibria of the game.
Then $f_i = g_i$ for $1 \le i \le N$.
Indeed, if 
$(f_1, \ldots, f_N) \in \mcC_1 \times  \ldots \times \mcC_N$
is a pure
Nash equilibrium of the game.
Then $f_i = \rho_i f$ for $1 \le i \le N$, where
where $f := \sum_{i=1}^N f_i$.
\end{thm}

\bpf
We prove the stronger statement. The proof is similar to that 
of Theorem \ref{thm:proportional2}.
Suppose player $i$ reacts to the
strategy profile $(f_j, j \neq i)$ by playing 
$\frac{\rho_i}{\sum_{j \neq i} \rho_j} \sum_{j \neq i} f_j$.
Let $\rho := \sum_{i=1}^N \rho_i$. 
We have 
\begin{eqnarray*}
&& V_i(\frac{\rho_i}{\sum_{j \neq i} \rho_j} \sum_{j \neq i} f_j, 
(f_j, j \neq i))\\
&& \stackrel{(a)}{=}
\frac{\rho_i}{\sum_{j \neq i} \rho_j}
\int_{x \in D} \sum_{j \neq i} f_j(x) \int_{y \in D} 
e^{ - \int_{u \in B(y \rightarrow x)} 
\frac{\rho}{\sum_{j \neq i} \rho_j}  \sum_{j \neq i} f_j(u) \eta(du)} \eta(dy) \eta(dx)\\
&& = \frac{\rho_i}{\rho} 
\int_{x \in D} \frac{\rho}{\sum_{j \neq i} \rho_j} \sum_{j \neq i} f_j(x) \int_{y \in D} 
e^{ - \int_{u \in B(y \rightarrow x)} 
\frac{\rho}{\sum_{j \neq i} \rho_j}  \sum_{j \neq i} f_j(u) \eta(du)} \eta(dy) \eta(dx)\\
&&\stackrel{(b)}{=} \frac{\rho_i}{\rho} \eta(D) (1 - e^{-\rho}),
\end{eqnarray*}
where step (a) is from Equation \eqref{eq:value} and 
step (b) is from Equation \eqref{eq:tautology}. 
Since $(f_1, \ldots, f_N)$ is a Nash equilibrium, it follows that 
\[
V_i(f_1, \ldots, f_N) \ge \frac{\rho_i}{\rho} \eta(D) (1 - e^{-\rho}).
\]
Since this holds for all $1 \le i \le N$, it follows from 
Equation \eqref{eq:conservationN} that this inequality must hold
with equality for all $1 \le i \le N$, i.e. that we have
\[
V_i(f_1, \ldots, f_N) = \frac{\rho_i}{\rho} \eta(D) (1 - e^{-\rho}),
\]
for all $1 \le i \le N$. But then, for each $1 \le i \le N$ we have
\[
V_i(f_i = \frac{\rho_i}{\sum_{j \neq i} \rho_j} \sum_{j \neq i} f_j, 
(f_j, j \neq i)) = V_i(f_1, \ldots, f_N),
\]
so from the strict concavity property of the value function of
player $i$ proved in Lemma \ref{lem:convexityN} and the 
fact that $f_i$ is a best response of player $i$ to the 
strategy profile $(f_j, j \neq i)$ of the other players, we must have
$f_i = \frac{\rho_i}{\sum_{j \neq i} \rho_j} \sum_{j \neq i} f_j$, 
which is the same as $f_i = \frac{\rho_i}{\rho} f$, where 
$f := \sum_{i=1}^N f_i$. This completes the proof of the theorem.

\epf

The properties of Nash equilibria of Cox process Hotelling games
established so far, assuming a Nash equilibrium exists, can be gathered
into the following statement.

\begin{thm} 		 \label{thm:ifNEexists}
Consider an $N$-player Cox process Hotelling game
$(D, \eta, \rho_1, \ldots, \rho_N)$.
Let $\rho := \sum_{i=1}^N \rho_i$.
The game admits a Nash equilibrium if and only if there is a
function $f \in \mcC(\rho)$ such that 
\begin{eqnarray}		\label{eq:dominate}
&& \int_{x \in D} f(x) \int_{y \in D} 
e^{ - \int_{u \in B(y \rightarrow x)} f(u) \eta(du)} \eta(dy) \eta(dx)
\nonumber \\
&&~~~~~~~ \ge
\int_{x \in D} g (x) \int_{y \in D} 
e^{ - \int_{u \in B(y \rightarrow x)} f(u) \eta(du)} \eta(dy) \eta(dx),
\end{eqnarray}
for all $g \in \mcC(\rho)$. If such a function exists, the game
has a unique Nash equilibrium, given by the pure strategy profile
$(\frac{\rho_1}{\rho} f, \ldots, \frac{\rho_N}{\rho} f)$.
\end{thm}

\bpf
Suppose first that the game admits a Nash equilibrium. 
By Theorem \ref{thm:proportionalN}, we know that this Nash
equilibrium is unique and is a pure Nash equilibrium of the form 
$(\frac{\rho_1}{\rho} f, \ldots, \frac{\rho_N}{\rho} f)$,
for some $f \in \mcC(\rho)$. Since $\frac{\rho_i}{\rho} f$
is a best response of player $i$ to the strategy profile
$(f_j = \frac{\rho_j}{\rho} f, j \neq i)$ of the other players, by the 
definition of Nash equilibrium, see Equation 
\eqref{eq:NashEqForPlayeri}, we must have
\begin{eqnarray*}
&& \int_{x \in D} \frac{\rho_i}{\rho} f \int_{y \in D} 
e^{ - \int_{u \in B(y \rightarrow x)} f(u) \eta(du)} \eta(dy) \eta(dx)\\
&&~~~~~~~ \ge
\int_{x \in D} g_i (x) \int_{y \in D} 
e^{ - \int_{u \in B(y \rightarrow x)} f(u) \eta(du)} \eta(dy) \eta(dx),
\end{eqnarray*}
for all $g_i \in \mcC(\rho_i)$, which is the same as 
the condition in Equation \eqref{eq:dominate}.

Conversely, suppose the condition in Equation \eqref{eq:dominate} holds
for some function $f \in \mcC(\rho)$. Then, by the definition of
Nash equilibrium in Equation \eqref{eq:NashEqForPlayeri}, 
we see that the strategy profile 
$(\frac{\rho_1}{\rho} f, \ldots, \frac{\rho_N}{\rho} f)$
is a Nash equilibrium of the game. The existence of such a Nash
equilibrium then guarantees, by Theorem \ref{thm:proportionalN},
that it is the unique Nash equilibrium of the game.
\epf

\subsection{Relationship with ordinal potential games}
\label{subsec:Potential}

Recall that an $n$-player game, with player $i$ having action set $\mcY_i$, is called an {\em ordinal potential game} \cite{MondShap}
if there is a function $P: \prod_{i=1}^n \mcY_i \to \mbbR$ such 
that,
for all $1 \le i \le n$, $y_i, z_i \in Y_i$, and
$(y_j, j \neq i) \in \prod_{j \neq i} Y_j$,
we have
$V_i(y_i, (y_j, j \neq i)) > V_i(z_i, (y_j, j \neq i))$ iff 
$P(y_i, (y_j, j \neq i)) > P(z_i, (y_j, j \neq i))$.

We have established that when a Cox process Hotelling game
admits a Nash equilibrium it admits a unique pure Nash equilibrium.
Since ordinal potential games are a well-known class of games that admit
pure strategy Nash equilibria, 
this leads naturally to the question of whether Cox process Hotelling games are 
ordinal potential games. 
However, it is possible to argue that in general this is not true.

To see this, consider the $2$-player Cox process Hotelling game between
Alice and Bob on $D$,
taken to be a circle of radius $1$ centered at the origin in $\mbbR^2$, the base measure $\eta$
being the Lebesgue measure on $D$. Thus $\eta(D) = 2 \pi$. Suppose that 
$\rho_A = \rho_B = \frac{\rho}{2}$, where $\rho$ should be thought of as
being sufficiently large in a sense that we will make precise shortly. Let 
$\epsilon > 0$ be sufficiently small 
(to be precise, we require that $\epsilon < \frac{2 \pi}{9}$).

We consider four pure strategies, i.e elements of $\mcC(\frac{\rho}{2})$, 
denoted by $\sigma^R$, $\beta^R$, $\sigma^L$ and $\beta^L$ 
respectively, defined as follows:

 $\sigma^R$ is constant over the arc of the circle of length $\epsilon$ centered
at $(1,0)$ and is zero elsewhere;

Consider the arc of the circle of length $\frac{\epsilon}{2}$ centered at 
$(\frac{\sqrt{3}}{2}, \frac{1}{2})$ and the arc of the circle of length $\frac{\epsilon}{2}$ centered at 
$(\frac{\sqrt{3}}{2}, - \frac{1}{2})$.
$\beta^R$ is constant over the union of these two arcs and is zero elsewhere;

$\sigma^L$ is constant over the arc of the circle of length $\epsilon$ centered
at $(-1,0)$ and is zero elsewhere. It can be considered to be the ``left" version
of $\sigma^R$, which is its ``right" version;

$\beta^L$ is the ``left" version of $\beta^R$, which is its ``right" version.
Namely, $\beta^L$ is uniform over the union of the two arcs of the circle
of length $\frac{\epsilon}{2}$ centered at 
$(\frac{- \sqrt{3}}{2},\frac{1}{2})$ and 
$(\frac{- \sqrt{3}}{2}, - \frac{1}{2})$ respectively,
and is zero elsewhere.


To be consistent with the convention in the rest of the document, we will use a subscript to indicate the identity of the player playing the strategy.
Thus, for instance, the strategy pair $(\sigma^L_A, \sigma^R_B)$ 
indicates that Alice is playing the strategy $\sigma^L$ and Bob is playing the
strategy $\sigma^R$.

It is straightforward to check that 
\[
\lim_{\rho \to \infty} V_A(\sigma^R_A, \beta^R_B) = \frac{\pi}{6} + \frac{\epsilon}{4},
\]
and 
\[
\lim_{\rho \to \infty} V_A(\sigma^L_A, \beta^R_B) = \frac{5 \pi}{6} + \frac{\epsilon}{4}.
\]
We will use these facts and their obvious consequences in the following argument.

Suppose there were an ordinal potential function $P: \mcC(\frac{\rho}{2}) \times 
\mcC(\frac{\rho}{2}) \to \mbbR$ for this $2$-player Cox process Hotelling game.

For sufficiently large $\rho$, we have 
$V_A(\sigma^R_A, \beta^R_B) < V_A(\sigma^L_A, \beta^R_B)$, and so
$P(\sigma^R, \beta^R) < P(\sigma^L, \beta^R)$. 

For sufficiently large $\rho$ we also have
$V_B(\sigma^L_A, \beta^R_B) < V_B(\sigma^L_A, \beta^L_B)$, and so
$P(\sigma^L, \beta^R) < P(\sigma^L, \beta^L)$.

But for sufficiently large $\rho$ we also have
$V_A(\sigma^L_A, \beta^L_B) < V_A(\sigma^R_A, \beta^L_B)$, and so
$P(\sigma^L, \beta^L) < P(\sigma^R, \beta^L)$. 

Finally, for sufficiently large $\rho$ we also have
$V_B(\sigma^R_A, \beta^L_B) < V_B(\sigma^R_A, \beta^R_B)$, and so
$P(\sigma^R, \beta^L) < P(\sigma^R, \beta^R)$.

Putting these together leads to a contradiction. Hence this $2$-player Cox
process Hotelling game is not an ordinal potential game.

\subsection{Nash equilibria may not exist}	\label{subsec:Nonexistence}

Consider a $2$-player Cox process Hotelling game
$(D,\eta,\rho_A,\rho_B)$.
The results of Theorems \ref{thm:purity2}, 
\ref{thm:pureandunique2} and \ref{thm:proportional2} are
consistent with that of 
Theorem \ref{thm:unique2} in the case where $D$ is compact,
admits a transitive group of metric-preserving automorphisms,
and $\eta$ is an invariant measure under the action of this group. 
This might lead one to expect that the constant intensity pair
$(\barrho_A, \barrho_B)$ is a Nash equilibrium for a
general $2$-player Cox process Hotelling game
$(D, \eta, \rho_A, \rho_B)$. The following simple example shows
that this is not the case.

\begin{exm}		\label{exm:unitinterval}
Let $D$ be the interval $[- \frac{1}{2}, \frac{1}{2}]$ of the real line,
with $\eta$ being the Lebesgue measure restricted to $D$. 
Then, for 
every $2$-player Cox process Hotelling game $(D,\eta,\rho_A,\rho_B)$,
the constant 
intensity pair $(\barrho_A, \barrho_B)$ is not a Nash
equilibrium of the game. 

To see this, first note that 
$\eta(D) = 1$, so $\barrho_A = \rho_A$ and $\barrho_B = \rho_B$.
Recall that $\rho := \rho_A + \rho_B$.
From Theorem \ref{thm:ifNEexists},
it suffices to find $g \in \mcC(\rho)$, i.e.  $g: [- \frac{1}{2}, \frac{1}{2}] \to \mbbR_+$ with $\int_{u = -\frac{1}{2}}^{\frac{1}{2}} g(u) du = \rho$,
such that
\begin{equation}		\label{eq:dominator}
\int_{x = -\frac{1}{2}}^{\frac{1}{2}} g(x) \int_{y = - \frac{1}{2}}^{\frac{1}{2}} 
e^{ - \rho \eta(B(y \rightarrow x))}  dy dx
>
\int_{x = -\frac{1}{2}}^{\frac{1}{2}}
\rho \int_{y = - \frac{1}{2}}^{\frac{1}{2}} 
e^{ - \rho \eta(B(y \rightarrow x))}  dy dx,
\end{equation}
where we have written $dx$ and $dy$ in the integrals instead of
$\eta(dx)$ and $\eta(dy)$ respectively, because $\eta$ is the
Lebesgue measure. By Equation \eqref{eq:tautology} the integral 
on the right hand side of Equation \eqref{eq:dominator} is 
$1 - e^{-\rho}$.
For the integral on the left hand side of 
Equation \eqref{eq:dominator}, let us first replace 
$g(x) dx$ by $\rho \delta_0(dx)$ where $\delta_0$ is the measure on 
$[- \frac{1}{2}, \frac{1}{2}]$ giving mass $1$ to the point at the 
origin, i.e. let us consider the integral 
\begin{eqnarray*}
&& \rho \int_{x = -\frac{1}{2}}^{\frac{1}{2}} \int_{y = - \frac{1}{2}}^{\frac{1}{2}} 
e^{ - \rho \eta(B(y \rightarrow x))}  dy \delta_0(dx)\\
&&~~~~~~~ = \rho \int_{- \frac{1}{2}}^{\frac{1}{2}} e^{ - \rho \eta(B(y \rightarrow 0))} dy\\
&&~~~~~~~ = 2 \rho \int_0^{\frac{1}{4}} e^{-2 \rho y} dy  
+ 2 \rho \int_{\frac{1}{4}}^{\frac{1}{2}} e^{-\frac{\rho}{2}} dy\\
&&~~~~~~~ 
= 1 - e^{-\frac{\rho}{2}} +
\frac{\rho}{2} e^{-\frac{\rho}{2}}.
\end{eqnarray*}
For all $\rho > 0$
this integral is strictly bigger than 
$1 - e^{ - \rho}$.
It follows that we can find 
$g \in \mcC(\rho)$ to get the strict inequality in Equation 
\eqref{eq:dominator}, as desired.
\hfill $\Box$



\end{exm}

In the $2$-player game considered in Example \ref{exm:unitinterval},
one can in fact conclude that there is no Nash equilibrium
when the sum of the intensities of the two players is sufficiently small.
We state this formally.

\begin{thm}		\label{thm:NoNE}
Let $D$ be the interval $[- \frac{1}{2}, \frac{1}{2}]$ of the real line,
with $\eta$ being the Lebesgue measure restricted to $D$. 
Then the $2$-player Cox process Hotelling game 
$(D,\eta,\rho_A,\rho_B)$ does not admit a Nash equilibrium
when $\rho := \rho_A + \rho_B < \log_e 4$.
\end{thm}

\bpf

Theorem \ref{thm:ifNEexists} tells us that, to prove that
a Nash equilibrium does not exist for this $2$-player game,
it suffices to show that for every 
$f \in \mcC(\rho)$ 
there is some $g \in \mcC(\rho)$ 
such that
\begin{equation}		\label{eq:gendom}
\int\limits_{x = -\frac{1}{2}}^{\frac{1}{2}} g(x) \int\limits_{y = - \frac{1}{2}}^{\frac{1}{2}} 
e^{ - \int_{ u \in B(y \rightarrow x)} f(u) du}  dy dx
>
\int\limits_{x = -\frac{1}{2}}^{\frac{1}{2}} f(x) \int\limits_{y = - \frac{1}{2}}^{\frac{1}{2}} 
e^{ - \int_{ u \in B(y \rightarrow x)} f(u) du}  dy dx.
\end{equation}

Let $f \in \mcC(\rho)$.
For $x \in [- \frac{1}{2}, \frac{1}{2}]$, define
$\psi_x ~:~ [- \frac{1}{2}, \frac{1}{2}] \to \mbbR_+$ via:
\[
\psi_x(y) := e^{ - \int_{ u \in B(y \rightarrow x)} f(u) du}.
\]
 Suppose
$\int_{y = - \frac{1}{2}}^{\frac{1}{2}} \psi_x(y) dy$ is not 
constant in $x$.
Let 
$x^* \in \mbox{arg max}_x \int_{y = - \frac{1}{2}}^{\frac{1}{2}} \psi_x(y) dy$, which exists because 
$\int_{y = - \frac{1}{2}}^{\frac{1}{2}} \psi_x(y) dy$ is continuous
in $x$ over 
$[- \frac{1}{2}, \frac{1}{2}]$. 
Since 
$$\int_{y = - \frac{1}{2}}^{\frac{1}{2}} \psi_x(y) dy$$ is not 
constant in $x$ over $x \in [- \frac{1}{2}, \frac{1}{2}]$, 
and since $\int_{x = - \frac{1}{2}}^{\frac{1}{2}} f(x) dx = \rho$,
this implies that
\[
\rho \int_{y = - \frac{1}{2}}^{\frac{1}{2}}  e^{ - \int_{ u \in B(y \rightarrow x^*)} f(u) du}dy
> \int_{x = -\frac{1}{2}}^{\frac{1}{2}} f(x) \int_{y = - \frac{1}{2}}^{\frac{1}{2}} 
e^{ - \int_{ u \in B(y \rightarrow x)} f(u) du}  dy dx,
\]
from which we can conclude the existence of 
$g \in \mcC(\rho)$ satisfying the strict inequality in 
Equation \eqref{eq:gendom}. 

From Theorem \ref{thm:ifNEexists}, 
we also know that if a Nash 
equilibrium exists it must be pure and of the form
$(\frac{\rho_A}{\rho} f, \frac{\rho_B}{\rho} f)$ 
for some $f \in \mcC(\rho)$. 
We claim that it must further be the case that $f(u) = f(-u)$
for all $u \in [- \frac{1}{2}, \frac{1}{2}]$, i.e. that 
$f$ is an {\emph {even}} function. This is because,
by symmetry, if $(\frac{\rho_A}{\rho} f, \frac{\rho_B}{\rho} f)$ 
is a Nash equilibrium
then so is $(\frac{\rho_A}{\rho} \tilde{f}, \frac{\rho_B}{\rho} \tilde{f})$ ,
where $\tilde{f}(u) := f(-u)$ (so we also have 
$\tilde{f} \in \mcC(\rho)$), and then, because the Nash equilibrium
is unique, it must be the case that $\tilde{f} = f$.

Thus it suffices to show that 
when $\rho < \log_e 4$ it is impossible to find an even function
$f \in \mcC(\rho)$ such that
$\int_{y = - \frac{1}{2}}^{\frac{1}{2}} \psi_x(y) dy$ is
constant in $x$. We will do this by establishing that for every 
even function $f \in \mcC(\rho)$ we have 
\begin{equation}		\label{eq:zero-and-half}
\int_{y = - \frac{1}{2}}^{\frac{1}{2}} \psi_{-\frac{1}{2}}(y) dy
< 
\int_{y = - \frac{1}{2}}^{\frac{1}{2}} \psi_0(y) dy.
\end{equation}

Observe that $\psi_0(y)$ is an even function
of $y \in [- \frac{1}{2}, \frac{1}{2}]$. Further,
we have $\psi_0(y) \ge e^{- \frac{\rho}{2}}$ for all 
$y \in [- \frac{1}{2}, \frac{1}{2}]$, and we have 
$\psi_0(y) = e^{- \frac{\rho}{2}}$ for 
$- \frac{1}{2} \le y \le - \frac{1}{4}$.

Observe also that $\psi_{-\frac{1}{2}}(y)$ is nonincreasing 
over $y \in [- \frac{1}{2}, \frac{1}{2}]$, with 
$\psi_{-\frac{1}{2}}(- \frac{1}{2}) = 1$,
 $\psi_{-\frac{1}{2}}(- \frac{1}{4}) = e^{- \frac{\rho}{2}}$, and
 $\psi_{-\frac{1}{2}}(y) = e^{-\rho}$ for $0 \le y \le \frac{1}{2}$.
 
 From these two sets of observations, we have 
 $\psi_{-\frac{1}{2}}(y) \ge \psi_0(y)$ for 
 $y \in [- \frac{1}{2}, -\frac{1}{4}]$ and 
 $\psi_{-\frac{1}{2}}(y) \le \psi_0(y)$ for 
 $y \in [- \frac{1}{4}, \frac{1}{2}]$. Further, we have 
 \[
 \int_{y = - \frac{1}{2}}^{- \frac{1}{4}} \left( \psi_{-\frac{1}{2}}(y) - \psi_0(y) \right) dy \le \frac{1}{4} (1 - e^{-\frac{\rho}{2}}),
 \]
 and 
 \[
 \int_{y = 0}^{\frac{1}{2}} \left( \psi_0(y) -\psi_{-\frac{1}{2}}(y) \right) dy \ge \frac{1}{2} ( e^{-\frac{\rho}{2}} - e^{-\rho}),
 \]
 while we also have 
 \[
 \int_{y = - \frac{1}{4}}^0 \left( \psi_0(y) -\psi_{-\frac{1}{2}}(y) \right) dy \ge 0.
 \]
From this we conclude that 
\[
 \int_{y = - \frac{1}{2}}^{\frac{1}{2}} \left( \psi_0(y) -\psi_{-\frac{1}{2}}(y) \right) dy \ge \frac{1}{2} ( e^{-\frac{\rho}{2}} - e^{-\rho}) - 
 \frac{1}{4} (1 - e^{-\frac{\rho}{2}}) > 0,
 \]
 if $0 < \rho < \log_e 4$, 
 which establishes the strict inequality in Equation 
 \eqref{eq:zero-and-half} and completes the proof.
 A more careful analysis will increase the range of $\rho$ for which one can prove that the game does not admit a Nash equilibrium.

\epf

\subsection{Restriced Cox process Hotelling games}	\label{subsec:Restricted}

A more insightful characterization than
the one in Theorem \ref{thm:ifNEexists} of when
Nash equilibria exist in 
Cox process Hotelling games
remains an interesting open problem.
The main technical difficulty in proving the existence of Nash equilibria in such games is that 
the action spaces of the individual players, which are of the
form $\mcC(\rho_i)$, where $\rho_i > 0$
is the intensity budget of player $i$, are not compact 
when 
endowed with the topology of 
weak convergence. Indeed, we have seen 
in Section \ref{subsec:Nonexistence} that
Nash equilibria may not exist in some cases.


To better understand the question of 
when Nash equilibria exist in such games,
we therefore propose to study a family of {\em restricted Cox Process Hotelling games},
where the action space of each individual player is now a compact
set.
The restrictions can be imposed in such a way that a unique Nash equilibria
will be guaranteed to exist in each such restricted game, 
it will be pure
and, if the restrictions imposed on the individual players are proportional in a sense made precise below, this unique pure Nash equilibrium will be of proportional form. Further, by varying the restriction, we can vary the compact
action space of each player 
in such a 
way that the union
over all such choices 
is the full space of allowed actions for that player, namely $\mcC(\rho_i)$
for player $i$ having a total intensity budget of $\rho_i$.
If a Nash equilibrium did exist for the original Cox process Hotelling game, then this guarantees
it would be discovered as the Nash equilibrium 
for some profile of restrictions on the 
action spaces of the individual players,
i.e. in one of the restricted Cox process Hotelling games that we consider.
Pursuing this direction, which we leave as a topic for future research, may give more insight into what the characterization in Theorem 
\ref{thm:ifNEexists} is actually saying.

To carry out this program, 
we first 
demonstrate, for
each $\rho > 0$, 
a family of compact subsets of 
$\mcC(\rho)$ whose union is
$\mcC(\rho)$. 
Recall that $\mcC(\rho)$ 
is a subset of
$L^1(\eta)$, where 
$f \in \mcC(\rho)$ is 
identified with the measure $f \eta$
on $D$, and $\mcC(\rho)$ endowed with the 
topology of weak convergence of measures
which it inherits as a subset of $\mcM(D)$. We will now consider
$L^1(\eta)$ with its weak topology
defined by considering it to be a Banach space with Banach dual $L^\infty(\eta)$,
see \cite{Rudin} or \cite[pg. 44]{Diestel}. To avoid confusion,
recall that the topology of weak convergence on $\mcC(\rho)$ is called the narrow topology in the theory of Banach spaces, see e.g. \cite[Vol. 1, pg. 176]{Bogachev}, and is weaker than the weak topology.

From the theorem of Dunford and Pettis,
\cite[Thm. 3]{Diestel}, \cite{DunfordPettis} the closure in the weak topology of a subset of $L^1(\eta)$ is compact in the weak topology iff it is 
uniformly integrable. Here we recall, 
see e.g. \cite[Vol. 1, Defn. 4.5.1.]{Bogachev}, \cite[pg. 41]{Diestel} that a subset $S \subseteq L^1(\eta)$ is called uniformly integrable if 
\[
\lim_{c \to \infty} \sup_{f \in S} 
\int_D |f(x)| 1(|f(x)| \ge c) \eta(dx) = 0.
\]
Further, by the theorem of 
de la Vall\'{e}e Poussin
\cite{Chafai},
\cite[Thm. 2]{Diestel}
for $S \subseteq L^1(\eta)$ to be uniformly integrable it is necessary and sufficient that there be a nondecreasing convex function $\Theta: \mbbR_+ \to \mbbR_+$, with $\Theta(0) = 0$ and 
$\lim_{x \to \infty} \frac{\Theta(x)}{x} = \infty$, such that 
\[
\sup_{f \in S} \int_D \Theta(|f(x)|) \eta(dx) < \infty.
\]

For $\rho > 0$, 
$\Theta: \mbbR_+ \to \mbbR_+$
a nondecreasing convex function with 
$\Theta(0) = 0$ and 
$\lim_{x \to \infty} \frac{\Theta(x)}{x} = \infty$, and
$0 < K < \infty$, we propose to consider the subset of $\mcC(\rho)$ defined by
\begin{equation}		\label{eq:CrhophiK}
\mcC(\rho,\Theta,K) := \{f:D \to \mbbR_+ \mbox{ s.t. } \int_D f(x) \eta(dx) = \rho, \int_D \Theta(f(x)) \eta(dx) \le K \}.
\end{equation}

It can be checked that $\mcC(\rho,\Theta,K)$ is a closed subset of $L^1(\eta)$ in the weak topology.
By the theorem of 
de la Vall\'{e}e Poussin,
$\mcC(\rho,\Theta,K)$ is uniformly integrable, so by the Dunford-Pettis theorem it is a compact subset of 
$L^1(\eta)$ in the weak topology.
Since the weak topology on $L^1(\eta)$ 
is stronger than the topology of weak convergence (i.e. the narrow topology)
on $L^1(\eta)$, $\mcC(\rho,\Theta,K)$
is compact in the topology of weak convergence on $L^1(\eta)$.

For every $f \in L^1(\eta)$ it can be checked that there is some 
nondecreasing convex function
$\Theta: \mbbR_+ \to \mbbR_+$
with $\Theta(0) = 0$ and 
$\lim_{x \to \infty} \frac{\Theta(x)}{x} = \infty$, such that 
$\int_D \Theta(|f(x)|) \eta(dx) < \infty$.
It follows that the union of 
$\mcC(\rho,\Theta,K)$ over all choices of 
$\Theta$ and $K$ equals $\mcC(\rho)$. 
We thus have a family of compact 
subsets of 
$\mcC(\rho)$ whose union is
$\mcC(\rho)$. 

From the convexity of $\Theta$ it is also straightforward to show that 
each $\mcC(\rho,\Theta,K)$ is a convex subset of $L^1(\eta)$. As a closed subset of $\mcM(D)$
in the topology of weak convergence, $\mcC(\rho,\Theta,K)$ is a Borel subset of $\mcM(D)$, so we are able discuss probability measures on $\mcC(\rho,\Theta,K)$.

By an $N$-player {\em restricted Cox process Hotelling game} we mean a game with player $i$,
for $1 \le i \le N$, having the intensity budget $\rho_i > 0$ and the space of pure actions
some $\mcC(\rho_i, \Theta_i, K_i)$, which, as we have seen, is a compact subset of $\mcC(\rho_i)$
in the topology of weak convergence. 
Here, for each $1 \le i \le N$,
$\Theta_i: \mbbR_+ \to \mbbR_+$
is a nondecreasing convex function
with $\Theta_i(0) = 0$ and 
$\lim_{x \to \infty} \frac{\Theta_i(x)}{x} = \infty$, and
$0 < K_i < \infty$. 
Suppose the individual players play the pure actions
$f_i \in \mcC(\rho_i, \Theta_i, K_i)$ for $1 \le i \le N$.
Let $\Phi_i$ be a Poisson point process on $D$ with intensity measure
$f_i \eta$,
with these processes being mutually independent for $1 \le i \le N$,
and let $\Phi := \sum_{i=1}^N \Phi_i$.
We think of the points of $\Phi_i$ as the points of player $i$, since these
result from the choice of $f_i$, which was made by that
player.
Then, as before, the value of player $i$, denoted
$V_i(f_1, \ldots, f_N)$, is given by
$$V_i(f_1, \ldots, f_N) := \mathbb{E} [\sum_{x\in \Phi_i} \eta(W_\Phi(x))].$$
The
expectation is with respect to the joint law of $(\Phi_j, 1 \le j \le n)$,
which are independent.
Here, by definition, a sum over an empty set is 0.

Consider an $N$-player restricted Cox process Hotelling
game where
player $i$ has the space of actions
$\mcC(\rho_i, \Theta_i, K_i)$. Any mixed strategy $N$-tuple in this game can be written as 
$(f_1(M_1), \ldots, f_N(M_N))$, where 
$(M_i \in \mcM_j,  1 \le i \le N)$ are independent
random variables representing the randomizations used by 
the individual players in implementing their randomized strategies, and
$f_i(m_i) \in \mcC(\rho_i, \Theta_i, K_i)$ is the 
choice of action of player $i$ in case the realization of
her random variable $M_i$ is $m_i$.
The value of player $i$ in such a mixed strategy is 
$\mbbE[ V_i(f_1(M_1), \ldots, f_N(M_N))]$, the expectation being taken
with respect to the joint distribution of $(M_j, 1 \le j \le N)$, which are
independent. 
The vector of independently randomized strategies 
\[
(f_1(M_1), \ldots, f_N(M_N)) \in 
\mcC(\rho_1, \Theta_1, K_1) \times  \ldots \times \mcC(\rho_N, \Theta_N, K_N)
\]
is called a {\em Nash equilibrium} of the game
if, for all $g_j \in \mcC(\rho_j, \Theta_j, K_j)$, $1 \le j \le N$, we
have, for all $1 \le i \le N$, 
\begin{equation}		\label{eq:RestNashEqForPlayeri}
\mbbE[V_i(f_1(M_1), \ldots, f_N(M_N))] \ge 
\mbbE[V_i(g_i, (f_j(M_j), j \neq i))].
\end{equation}
Since each $\mcC(\rho_i, \Theta_i, K_i)$ is compact and each 
$V_i: \prod_{j=1}^N \mcC(\rho_j, \Theta_j, K_j) \to \mbbR_+$ is continuous, the existence of a mixed strategy Nash equilibrium for every restricted $N$-player Cox process Hotelling game is guaranteed \cite{Glicksberg}.

We now formally state and prove that every
restricted Cox process Hotelling game
has unique Nash equilibrium, and that this is comprised of pure strategies. The following
Theorem is a combination of
the analog of 
Theorem \ref{thm:purityN}
and the $N$-player version of 
Theorem \ref{thm:pureandunique2}
for restricted Cox process Hotelling games.

\begin{thm} 		 \label{thm:Restrpureandunique}
Consider an $N$-player restricted Cox process Hotelling game on the Polish space $D$ with base measure $\eta$ where player $i$,
for $1 \le i \le N$, has the intensity budget $\rho_i > 0$ and the space of pure actions
$\mcC(\rho_i, \Theta_i, K_i)$, where
$\Theta_i: \mbbR_+ \to \mbbR_+$
is a nondecreasing convex function
with $\Theta_i(0) = 0$ and 
$\lim_{x \to \infty} \frac{\Theta_i(x)}{x} = \infty$, and
$0 < K_i < \infty$. 
Suppose $(f_1(M_1), \ldots, f_N(M_N))$
is a Nash equilibrium of the game, which we know exists. Then there exist
$g_i \in \mcC_i$, $1 \le i \le N$, such that 
\[
\mbbP((f_1(M_1), \ldots, f_N(M_N) = (g_1, \ldots, g_N)) = 1,
\]
i.e. the Nash equilibrium is a pure strategy Nash equilibrium. Furthermore, if $(g_1, \ldots, g_N)$ and $(h_1, \ldots, h_N)$ are two pure strategy Nash equilibria for the game, then 
$g_i = h_i$ for all $1 \le i \le N$, i.e. 
the Nash equilibrium is unique.
\end{thm}

\bpf
The proof is similar to those of
the $N$-player versions of 
Theorem \ref{thm:purity2},
and Theorem \ref{thm:pureandunique2},
with the obvious modifications. 
All that is being used in those proofs is the 
convexity of the set of allowed actions of each player and the strict concavity of the value function of each player in its own action when the actions of its opponents are fixed. These properties continue to hold in the restricted Cox process Hotelling games that we are now considering.
\epf

When the restrictions on the individual 
players in an $N$-player restricted Cox process Hotelling game are proportional to their allowed intensities
we can characterize the Nash equilibrium of the game, which we known by Theorem 
\ref{thm:Restrpureandunique} is unique and comprised of pure strategies, in a manner analogous to what was done in Theorem \ref{thm:proportionalN}. To define what
we mean by proportional restrictions,
given $\alpha > 0$ and 
a nondecreasing convex function
$\Theta: \mbbR_+ \to \mbbR_+$
with $\Theta(0) = 0$ and 
$\lim_{x \to \infty} \frac{\Theta(x)}{x} = \infty$, we define the function 
$\Theta^{(\alpha)}: \mbbR_+ \to \mbbR_+$
by 
\[
\Theta^{(\alpha)}(x) := \Theta(\frac{x}{\alpha}),~~x \in \mbbR_+.
\]
Note that $\Theta^{(\alpha)}: \mbbR_+ \to \mbbR_+$ is a nondecreasing convex function with $\Theta^{(\alpha)}(0) = 0$ and
$\lim_{x \to \infty} \frac{\Theta^{(\alpha)}(x)}{x} = \infty$.
We can then make the following simple observation.

\begin{lem} \label{lem:proportional}
Given $\rho > 0$, 
a nondecreasing convex function
$\Theta: \mbbR_+ \to \mbbR_+$
with $\Theta(0) = 0$ and 
$\lim_{x \to \infty} \frac{\Theta(x)}{x} = \infty$, and $0 < K < \infty$, we have
$f \in \mcC(\rho, \Theta^{(\rho)}, K)$ 
iff $\frac{f}{\rho} \in \mcC(1, \Theta, K)$.
\end{lem}

\bpf
To check if $\frac{f}{\rho} \in \mcC(1, \Theta, K)$ we need to check if 
$\int_D \frac{f(x)}{\rho} \eta(dx) = 1$
and if $\int_D \Theta( \frac{f(x)}{\rho}) \eta(dx) \le K$. Equivalently, we need to check whether
$\int_D f(x) \eta(dx) = \rho$ and
$\int_D \Theta^{(\rho)}(f(x)) \eta(dx) \le K$,
i.e. whether $f \in \mcC(\rho, \Theta^{(\rho)}, K)$.
\epf

The following result characterizes the Nash equilibria of $N$-player restricted Cox process Hotelling games when the restrictions on the individual 
players are in proportion to their allowed intensities.

\begin{thm} 		 \label{thm:Restrproportional}
Consider an $N$-player restricted Cox process Hotelling game on the Polish space $D$ with base measure $\eta$ where player $i$,
for $1 \le i \le N$, has the intensity budget $\rho_i > 0$ and the space of pure actions
$\mcC(\rho_i, \Theta^{(\rho_i)}, K)$, where
$\Theta: \mbbR_+ \to \mbbR_+$
is a nondecreasing convex function
with $\Theta(0) = 0$ and 
$\lim_{x \to \infty} \frac{\Theta(x)}{x} = \infty$, and
$0 < K < \infty$. 
Then the game has a unique Nash equilibrium,
which is of the form $(f_1, \ldots, f_N)$,
where $f_i = \frac{\rho_i}{\rho} f$ for some
$f \in \mcC(\rho, \Theta^{(\rho)}), K)$, where
$\rho := \sum_{i=1}^N \rho_i$.
\end{thm}

\bpf
The proof is similar to that of Theorem 
\ref{thm:proportionalN}, with the obvious modifications. The only thing that needs to be observed is that for any choice of $f_i \in \mcC(\rho_i, \Theta^{(\rho_i)}, K)$ for $1 \le i \le N$, we also have: 

(i)
$\frac{\rho_i}{\sum_{j \neq i} \rho_j} \sum_{j \neq i} f_j \in \mcC(\rho_i, \Theta^{(\rho_i)}, K)$ for all $1 \le i \le N$; 

(ii) $\sum_{i=1}^N f_i =: f 
\in \mcC(\rho, \Theta^{(\rho)}, K)$, and;

(iii) $\frac{\rho_i}{\rho} f 
\in \mcC(\rho_i, \Theta^{(\rho_i)}, K)$ for all $1 \le i \le N$.

To show (i), by Lemma \ref{lem:proportional}
what we need to show, for all $1 \le i \le N$,
is that 
$\frac{1}{\sum_{j \neq i} \rho_j} \sum_{j \neq i} f_j \in \mcC(1, \Theta, K)$. 
By Lemma \ref{lem:proportional} again, we have
$\frac{f_j}{\rho_j} \in \mcC(1, \Theta, K)$
for all $j \neq i$. The desired claim follows from the convexity of $\mcC(1, \Theta, K)$.

To show (ii), by Lemma \ref{lem:proportional}
what we need to show is that $\frac{f}{\rho} 
\in \mcC(1, \Theta, K)$. By Lemma \ref{lem:proportional} we have $\frac{f_i}{\rho_i} \in \mcC(1, \Theta, K)$ for all $1 \le i \le N$.
The desired claim follows from the convexity of $\mcC(1, \Theta, K)$.

As for (iii), it is an immediate consequence of 
Lemma \ref{lem:proportional} since we have established the claim in (ii).
\epf

Finally, with a view to using restricted Cox process Hotelling games as a vehicle for better understanding the characterization in 
Theorem \ref{thm:ifNEexists}, we can 
prove the following analog of that result for
restricted Cox process Hotelling games.

\begin{thm} 		 \label{thm:RestrifNEexists}
Consider an $N$-player restricted Cox process Hotelling game on the Polish space $D$ with base measure $\eta$ where player $i$,
for $1 \le i \le N$, has the intensity budget $\rho_i > 0$ and the space of pure actions
$\mcC(\rho_i, \Theta^{(\rho_i)}, K)$, where
$\Theta: \mbbR_+ \to \mbbR_+$
is a nondecreasing convex function
with $\Theta(0) = 0$ and 
$\lim_{x \to \infty} \frac{\Theta(x)}{x} = \infty$, and
$0 < K < \infty$. 
The game admits a unique Nash equilibrium,
comprised of pure strategies,
defined in terms of a 
function $f \in \mcC(\rho, \Theta^{(\rho)}, K)$,
where $\rho := \sum_{i=1}^N \rho_i$, such that 
\begin{eqnarray}        \label{eq:Restrcriterion}		
&& \int_{x \in D} f(x) \int_{y \in D} 
e^{ - \int_{u \in B(y \rightarrow x)} f(u) \eta(du)} \eta(dy) \eta(dx)
 \nonumber \\
&&~~~~~~~ \ge
\int_{x \in D} g (x) \int_{y \in D} 
e^{ - \int_{u \in B(y \rightarrow x)} f(u) \eta(du)} \eta(dy) \eta(dx),
\end{eqnarray}
for all $g \in \mcC(\rho, \Theta^{(\rho)}, K)$. 
The corresponding unique Nash equilibrium of the 
game is 
given by the pure strategy profile
$(\frac{\rho_1}{\rho} f, \ldots, \frac{\rho_N}{\rho} f)$.
\end{thm}

\bpf
The proof is similar to that of Theorem 
\ref{thm:ifNEexists}, with the obvious modifications. The key point we need to observe
is that, for all $1 \le i \le N$, we have 
$g \in \mcC(\rho, \Theta^{(\rho)}, K)$
iff $\frac{\rho_i}{\rho} g 
\in \mcC(\rho_i, \Theta^{(\rho_i)}, K)$.
Also, as we have observed earlier, in view of the compactness of the action spaces of the individual players and the continuity of the payoff of an individual player in its action when the actions of its opponents are fixed, the existence of 
a mixed strategy Nash equilibrium is 
guaranteed \cite{Glicksberg}.
\epf

\begin{rem}
For the $N$-player restricted Cox process
Hotelling games with proportional restrictions
of the kind considered in Theorem \ref{thm:RestrifNEexists} 
that theorem tells us that for each
choice of $\rho_i > 0$ for $1 \le i \le N$,
nondecreasing convex function
$\Theta: \mbbR_+ \to \mbbR_+$ satisfying
$\Theta(0) = 0$ and 
$\lim_{x \to \infty} \frac{\Theta(x)}{x} = \infty$,
and $0 < K < \infty$, there is a 
unique function 
$f \in \mcC(\rho, \Theta^{(\rho)}, K)$
which satisfies the inequality in equation
\eqref{eq:Restrcriterion}
for all $g \in \mcC(\rho, \Theta^{(\rho)}, K)$,
where $\rho := \sum_{i=1}^N \rho_i$.
We also know that if, for the given choices
of $\rho_i > 0$ for $1 \le i \le N$, the original
$N$-player Cox player Hotelling game considered
in Theorem \ref{thm:ifNEexists}
admits a Nash equilibrium 
then there will be a function $f \in \mcC(\rho)$
satisfying the inequality in Equation 
\eqref{eq:dominate} for all $g \in \mcC(\rho)$,
and, most importantly, that this $f$ will 
be manifest itself as the one verifying
Equation \eqref{eq:Restrcriterion} for some 
choices of nondecreasing convex function
$\Theta: \mbbR_+ \to \mbbR_+$ satisfying
$\Theta(0) = 0$ and 
$\lim_{x \to \infty} \frac{\Theta(x)}{x} = \infty$,
and of $0 < K < \infty$.
It is in this sense that the discussion 
of $N$-player restricted Cox process Hotelling
games gives a vehicle, in principle, to better understand the meaning of the criterion
in Equation \eqref{eq:dominate} for the existence of Nash equilibria in $N$-player Cox process Hotelling games.
\end{rem}








\section*{Acknowledgements}

VA acknowledges support from the NSF grants CNS--1527846, CCF--1618145, CCF-1901004, the NSF Science \& Technology Center grant CCF--0939370 (Science of Information), and the William and Flora Hewlett Foundation
supported Center for Long Term Cybersecurity at Berkeley.
FB was supported by the ERC NEMO grant, under the European Union's Horizon 2020 research and innovation programme,
grant agreement number 788851 to INRIA.
The authors thank Prof. Yadati Narahari for posing the question of whether Cox process Hotelling games might be potential games.










\section*{Appendix}

\appendix

\section{A sufficient condition for $f \eta$ to be non-conflicting} \label{sec:append}

We give here a sufficient condition for the $\eta$-measure of
the boundary of the Voronoi tessellation of a Poisson point
process $\Phi$ of density $f$ w.r.t. $\eta$ on $D$ to be zero a.s. The setting is that of Subsection \ref{subsec:diffuse}, with 
$\int_D f(x)\eta(dx)=\rho <\infty$,
so that $\Phi$ has a finite number of points a.s.

A sufficient condition for the desired property to hold is that
$$\eta\{ z\in D \mbox{ s.t. } \exists X\ne Y\in \Phi \mbox{ with }
d(z,X)=d(z,Y)\} =0,\quad \mbox{a.s.},$$
which holds if
$$\mathbb{E}[\eta\{ z\in D \mbox{ s.t. } \exists X\ne Y\in \Phi \mbox{ with }
d(z,X)=d(z,Y)\}] =0.$$
The latter can be written as
$$\sum_{n\ge 2} \frac{\rho^n}{n !} e^{-\rho} n(n-1)
\int_{z\in D} \int_{x\in D} \int_{y\in D}
1_{d(z,x)=d(z,y)} \eta(dz) f(x) \eta(dx) f(y)\eta(dy)=0,$$
so that a sufficient condition for the desired property to hold is that
\begin{equation}\int_{z\in D} \int_{x\in D} \int_{y\in D}
1_{d(z,x)=d(z,y)} \eta(dz) f(x) \eta(dx) f(y)\eta(dy)=0.\end{equation}

\end{document}